\documentclass[11pt]{amsart}

\usepackage[a4paper]{geometry}   

\usepackage[T1]{fontenc}
\usepackage[utf8]{inputenc}

\usepackage{mathtools} 

\usepackage{tikz-cd} 
\usepackage{array}
\usepackage{multirow}

\usepackage{booktabs} 

\usepackage{hyperref} 
\usepackage{cleveref} 

\usepackage{microtype}

\theoremstyle{plain}
\newtheorem{thm}{Theorem}[section]
\newtheorem*{thm*}{Theorem}
\newtheorem{lemme}[thm]{Lemma}
\newtheorem{prop}[thm]{Proposition}
\newtheorem{cor}[thm]{Corollary}

\newtheorem*{conj*}{Conjecture}
\newtheorem{thmintro}{Theorem} 
\newtheorem{propintro}[thmintro]{Proposition}

\theoremstyle{definition}
\newtheorem{defn}[thm]{Definition}
\newtheorem*{defn*}{Definition}

\newtheorem*{question*}{Question}

\theoremstyle{remark}
\newtheorem{rem}[thm]{Remark}

\DeclarePairedDelimiter\abs{\lvert}{\rvert} 

\newcommand*\conjug[1]{\overline{#1}} 
\renewcommand{\Re}{\operatorname{Re}}
\renewcommand{\Im}{\operatorname{Im}}
\renewcommand*\setminus{\ensuremath{{}-{}}}

\newcommand{\sslash}{/\!\!/} 

\newcommand{\HH}{\mathbb{H}}
\newcommand{\CC}{\mathbb{C}}
\newcommand{\RR}{\mathbb{R}}
\newcommand{\ZZ}{\mathbb{Z}}

\newcommand{\NN}{\mathbb{N}}
\newcommand{\galCR}{\mathrm{Gal}(\CC/\RR)}

\DeclareMathOperator{\Hom}{Hom} 
\DeclareMathOperator{\Span}{Span} 
\DeclareMathOperator{\ev}{ev} 
\DeclareMathOperator{\tr}{tr} 
\DeclareMathOperator{\Ad}{Ad}

\author{Miguel Acosta}
\address{
	Département de Mathématiques, 
	Université du Luxembourg, 
	Maison du Nombre, 
	6, Avenue de la Fonte, 
	L-4364 Esch-sur-Alzette, 
	Luxembourg.}
\email{miguel.acosta@uni.lu}
\thanks{Miguel Acosta was partially supported by the grants R-STR-8023-00-B "MnLU-MESR CAFE-AutoFi" and R-AGR-3172-10-C "FNR-OPEN".} 

\title{On GIT quotients and real forms}
\date{\today}  



\begin{document}

\begin{abstract}
	We consider actions of complex algebraic groups $\mathbf{G}$ on complex algebraic varieties $\mathbf{X}$, coming from actions of real forms $G$ of $\mathbf{G}$ and $X$ of $\mathbf{X}$.
	We explore the links between the real points of the complex GIT quotient $\mathbf{X /\!\!/ G}$ and the real GIT quotient $X /\!\!/ G$ defined by Richardson and Slodowy.
	We prove that some type of real points of $\mathbf{X /\!\!/ G}$ can be lifted to a quotient of the form $X /\!\!/ G$ maybe after changing the real forms, and we link the number of possible lifts to a co-homology set. We apply then the results to character varieties, and study the particular case of the $\mathrm{SL}_3(\mathbb{C})$-character variety for $\mathbb{Z}$.
\end{abstract}

\maketitle

\section{Introduction}

 Needless to say, group actions are a very common and powerful tool in mathematics. 
 Given an action of a group on some space, a question rising naturally is to understand the space of its orbits.
 A geometric situation where this study is essential is given by the Ehresmann-Thurston principle in the setting of $(G,X)$-structures: deformation spaces of such structures on a manifold $M$ are locally parametrized by representations $\rho \in \Hom(\pi_1(M) , G)$ considered up to $G$-conjugation (see for example \cite{bergeron_gelander} or \cite{canary_epstein_green}). The corresponding spaces of orbits, obtained as quotients of $\Hom(\pi_1(M),G)$ by $G$ are the so-called \emph{character varieties}, and are a crucial tool for studying geometric structures.
 
 In a more general setting, given an action of a group $G$ on a space $X$, we would like to understand the space of $G$-orbits of $X$: the \emph{quotient} of $X$ by $G$. 
 In this article, we choose to work with a complex algebraic group $\mathbf{G}$ acting on a complex algebraic variety $\mathbf{X}$, as well as analog real objects $G$ and $X$.
 In this context, we have several additional structures that help to understand the space of orbits, since our spaces $\mathbf{X}$ (resp.\ $X$) are Hausdorff topological spaces and algebraic varieties.
 Thus, in addition to the categories of sets and topological spaces, we can consider the quotient of $\mathbf{X}$ by $\mathbf{G}$ in the setting of other categories. We focus mainly in the categories $\mathsf{Htop}$ of Hausdorff topological spaces and $\mathsf{AffAlg}$ of complex affine varieties.
 
 On the one hand, the quotient topology on the usual quotient set $\mathbf{X/G}$ is not Hausdorff in general, but there is still always a well defined quotient of $\mathbf{X}$ by $\mathbf{G}$ in $\mathsf{Htop}$. It is obtained by considering the \emph{Hausdorffization} of $\mathbf{X/G}$, that is its largest Hausdorff quotient.
 On the other hand, if $\mathbf{G}$ is a complex reductive group, the Geometric Invariant Theory (GIT) allows to consider the quotient of $\mathbf{X}$ by $\mathbf{G}$ in $\mathsf{AffAlg}$ as well. In that case, the GIT-quotient coincides with the $\mathsf{Htop}$ quotient, and is denoted $\mathbf{X\sslash G}$. We recall some general facts in \Cref{sect:background_and_setting}, like correspondence between points of $\mathbf{X \sslash G}$ and closed $\mathbf{G}$-orbits in $\mathbf{X}$, but refer to the notes of Brion \cite{brion_introduction_2010} for a more detailed exposition.

 As for the action of a real group $G$ on a real algebraic variety $X$, the algebraic quotients are no longer well defined. Nonetheless, Richardson and Slodowy construct in \cite{richardson_minimum_1990} a real GIT quotient $X \sslash G$, that is still the quotient of $X$ by $G$ in $\mathsf{Htop}$, and whose points are the closed $G$-orbits of $X$.
 In the case of character varieties, where $X = \mathrm{Hom}(\Gamma , G)$ for some finitely generated group $\Gamma$ and $G$ acts by conjugation, Parreau generalizes in \cite{parreau_espaces_2011} this construction to reductive groups $G$ over arbitrary local fields.
 
 If $G$ and $X$ are real forms of the complex group $\mathbf{G}$ and the complex variety $\mathbf{X}$, we have therefore two related real objects. First, we have the real GIT-quotient $X \sslash G$ constructed by Richardson and Slodowy in \cite{richardson_minimum_1990}, that is a Hausdorff space. Then, we can consider the real points of the complex GIT quotient $\mathbf{X \sslash G}$, that is defined over $\RR$. We know, thanks to Richardson and Slodowy, that there is a natural map $\beta \colon X \sslash G \to (\mathbf{X \sslash G})(\RR)$ that is proper and has finite fibers, but is not always surjective.
 The aim of this article is to have a better understanding of the links between the two quotients $X \sslash G$ and $\mathbf{X \sslash G}$.
 We call a point of $\mathbf{X \sslash G}$ \emph{neat} if it has a lift in $\mathbf{X}$ which is regular, has a closed $\mathbf{G}$-orbit and has stabilizer $Z(\mathbf{G})$. 
 We obtain the following result, that allows to lift \emph{neat} points of $(\mathbf{X \sslash G})(\RR)$ to quotients of the form $X \sslash G$, but maybe after changing the real forms $X$ and $G$.

 \begin{thmintro}[\Cref{thm:lift_real_irred}] \label{thm:intro_lifting_irred}
 	Let $\mathbf{G}$ be a connected complex reductive algebraic group, acting on a complex affine variety $\mathbf{X}$ with a trivial action of its center.
 	Let $G$ and $X$ be compatible real forms of $\mathbf{G}$ and $\mathbf{X}$ respectively.
 	Let $x \in \mathbf{X}$ be a neat point such that $\pi(x) \in (\mathbf{X \sslash G})(\RR)$.
 	Let $\mathbf{Y} \subset \mathbf{X}$ be the irreducible component of $x$.
 	Then, there exist real forms $G'$ of $\mathbf{G}$ and $Y'$ of $\mathbf{Y}$ such that
 	$G'$ is a twisted form of $G$,
 	$x \in Y'$ and $Y'$ is $G'$-stable.
 \end{thmintro}
 
 Moreover, the following proposition, that is mainly a consequence of Corollary 1 of \cite[Ch.\ I §5.4 prop.\ 36]{serre_cohomologie_1994d}, gives a bound on the cardinal of the fibers of the map $\beta \colon X \sslash G \to (\mathbf{X \sslash G})(\RR)$ in terms of cohomology sets with coefficients in the stabilizer of points of $\mathbf{X}$.
 Note that \Cref{prop:intro_bound_number_of_lifts} is proven for actions on vector bundles over surfaces in \cite[Proposition 2.4]{schaffhauser_finite_2016} and for moduli spaces of quivers over perfect fields in \cite[Proposition 3.3]{hoskins_rational_2017}.
 
 \begin{propintro}[\Cref{prop:bound_number_of_lifts}]\label{prop:intro_bound_number_of_lifts}
 	Let $x \in (\mathbf{X \sslash G})(\RR)$, $\tilde{x} \in \mathbf{X}$ be a lift of $x$ and $\mathbf{G}_{\tilde{x}}$ the stabilizer of $\tilde{x}$ in $\mathbf{G}$. Then,
 	there is a one-to-one correspondence between $\beta^{-1}(x)$ and the kernel of the natural map
 	\[ 
 	H^1(\galCR ,\mathbf{G}_{\tilde{x}} ) \to H^1(\galCR ,\mathbf{G} ).
 	 \]
 \end{propintro}

 We can apply these results to our original motivation for the study of quotients of the form $X \sslash G$ and $\mathbf{X \sslash G}$, namely the character varieties, defined, for a finitely generated group $\Gamma$, as
 \[ \mathfrak{X}_\mathbf{G}(\Gamma)
 =
 \Hom(\Gamma , \mathbf{G}) \sslash \mathbf{G}. \]
 The link between the two quotients has already been studied in particular cases. 
 For example, Johnson and Millson study in \cite{johnson_deformation_1987} the real points of the character variety $\mathfrak{X}_{\mathrm{SO}(n+2,\CC)}(\Gamma)$, where $\Gamma$ is a surface group and the corresponding real form of $\mathrm{SO}_{n+2}(\CC)$ is $\mathrm{SO}(n+1,1)$. In particular, they prove that the map $\beta$ is one-to-one when restricted to quasi-Fuchsian representations\footnote{
  In \cite{johnson_deformation_1987}, given a surface group representation $\rho_0 \colon \Gamma \to \mathrm{SO}(n+1,1)$ that factors through $\mathrm{SO}(2,1)$, a representation $\rho \in \Hom(\Gamma , \mathrm{SO}(n+1,1) )$ is called \emph{quasi-Fuchsian} if it is topologically conjugate to $\rho_0$.
 	}.
 More recently, Casimiro, Florentino, Lawton and Oliveira prove in	\cite{casimiro_topology_2014} that there is a deformation retraction between $G$-character varieties for free groups $\mathbb{F}_r$ and the corresponding quotients $\Hom(\mathbb{F}_r , K ) /K$ where $K$ is the maximal compact subgroup of $G$, and observe the links between real points of a complex character variety and the points coming from a real form in some particular examples.
 The author has also studied the particular cases of real points of $\mathrm{SL}_n(\CC)$-character varieties in \cite{acosta_character_2019a}, as well as for classical complex groups in \cite{acosta_character_2019}.
 Here, we obtain the following result, that generalizes all those cases. 
 
 \begin{thmintro}[\Cref{thm:lift_real_characters}]\label{thm:intro_lift_real_characters}
 	Let $\Gamma$ be a finitely generated group and
 	$\mathbf{G}$ a connected complex reductive algebraic group.
 	Let $\Phi$ be an antiholomorphic involution of $\mathfrak{X}_{\mathbf{G}}(\Gamma)$ induced by a real form $G$ of $\mathbf{G}$.
 	Let $\rho \in \Hom(\Gamma , \mathbf{G})$ be a good representation such that $\Phi(\chi_\rho) = \chi_\rho$.
 	Then, there exists a real form $G'$ of $\mathbf{G}$ such that 
 	$\rho \in \Hom(\Gamma,G')$.
 \end{thmintro}

 Finally, we study the particular case of the $\mathrm{SL}_3(\CC)$-character variety of $\ZZ$, with its three real forms $\mathrm{SL}_3(\RR)$, $\mathrm{SU}(3)$ and $\mathrm{SU}(2,1)$. It is a relatively simple case but has a particular behavior, since we find situations where the cardinal of the fibers given by \Cref{prop:intro_bound_number_of_lifts} take different values, and the real GIT quotient that is not a manifold.

\subsection*{Outline of the article}
The article is organized as follows. First, we recall some background and fix notation and the setting in \Cref{sect:background_and_setting}. We focus on algebraic actions and real forms of complex groups and varieties; we also recall the setting for complex and real GIT quotients.
In \Cref{section:lifting_real_points}, we consider the problem of lifting real points of the complex GIT quotient to the corresponding real GIT quotient, study the number of lifts in each fiber and prove \Cref{thm:intro_lifting_irred} and \Cref{prop:intro_bound_number_of_lifts}.
Then, in \Cref{section:char-var}, we apply the previous results to character varieties and obtain \Cref{thm:intro_lift_real_characters}.
Finally, in \Cref{section:example}, we study in detail a particular example, namely the $\mathrm{SL}_3(\CC)$-character variety of $\ZZ$, where $\mathrm{SL}_3(\CC)$ acts on itself by inner automorphisms.

\subsection*{Acknowledgements}
The author would like to thank Pierre-Emmanuel Chaput for a long and helpful discussion about the background on $\mathbf{G}$-varieties and real forms and many comments on the article. He would like to thank as well Michel Brion, Gerald Schwarz, Ralph Bremigan, Mikhail Borovoi and Florent Schaffhauser for a very short but fruitful correspondence.

\section{Background and setting}\label{sect:background_and_setting}
 In this article, we consider vector spaces, algebraic groups and algebraic varieties that can be either real or complex.
 We will use a bold font for complex objects and the usual font for real ones. This way, $\mathbf{V}$, $\mathbf{X}$ and $\mathbf{G}$ will be a complex vector space, a complex affine variety and a complex algebraic group, while $V$, $X$ and $G$ will be real forms of them.
 Our setting will be of an action of a complex algebraic group $\mathbf{G}$ on an affine algebraic variety $\mathbf{X}$, coming from an action of a real form $G$ of $\mathbf{G}$ on a real form $X$ of $\mathbf{X}$. In this section, we fix notation for these objects and settle our preferred viewpoint on them.
 For the algebraic actions and GIT quotients, we refer to the notes of Brion \cite{brion_introduction_2010}, whereas for the setting on real forms we follow the book of Onishchik and Vinberg \cite{onishchik_lie_1990a}. Finally, we describe the real GIT quotient constructed by Richardson and Slodowy in \cite{richardson_minimum_1990}.

\subsection{Algebraic actions}

 An \emph{affine algebraic variety}, or simply \emph{variety}, is a subset $\mathbf{X}$ of some $\CC^n$ defined by polynomial equations. The maps between two algebraic varieties that are given by polynomials in the coordinates are called \emph{algebraic maps}.
 In this way, two varieties are \emph {isomorphic} if there is an algebraic isomorphism between them.
 Thus, we can consider abstract algebraic affine varieties. However, in this article, we will always think of an affine variety as embedded in some complex vector space $\mathbf{V}$, that can change depending on our needs.
 Note that, in this setting, we do not ask affine varieties to be irreducible. We say that a complex affine variety $\mathbf{X}$ is \emph{defined over $\RR$} if there is an embedding of $\mathbf{X}$ in a complex vector space $\mathbf{V}$ and there is a set of real polynomials defining its image in $\mathbf{V}$. Equivalently, $\mathbf{X}$ is defined over $\RR$ if there is an antiholomorphic involutive automorphism of $\mathbf{X}$.
 
 We denote the ring of functions of a variety $\mathbf{X}$ by $\CC[\mathbf{X}]$. If $\mathbf{X}$ is embedded in a complex vector space $\mathbf{V}$, then $\CC[\mathbf{X}]$ is the $\CC$-algebra obtained by restricting to $\mathbf{X}$ the polynomial functions on $\mathbf{V}$. Recall that affine varieties are in one-to-one correspondence with their rings of functions, so $\mathbf{X}$ is determined up to isomorphism by $\CC[\mathbf{X}]$.
  
 A \emph{complex algebraic group} is a group $\mathbf{G}$ that is also an affine variety, and such that the inverse and the group product are algebraic functions. 
 An action of $\mathbf{G}$ on a variety $\mathbf{X}$ is \emph{algebraic} if the corresponding map $\mathbf{G} \times \mathbf{X} \to \mathbf{X}$ is algebraic.
 If $\mathbf{V}$ is a complex vector space, and $\mathbf{G}$ acts linearly, the action is called a \emph{representation} of $\mathbf{V}$. If, in addition, every point is in a finite-dimensional $\mathbf{G}$-subspace, then $\mathbf{V}$ is called a \emph{$\mathbf{G}$-module}.
 We recall the following proposition that states that, in order to study the actions of algebraic groups on varieties, we only need to understand the $\mathbf{G}$-modules.
 
\begin{prop}[Proposition 1.9 of \cite{brion_introduction_2010}]\label{prop:equivariant_immersion}
 Let $\mathbf{G}$ be an affine algebraic group acting on a variety $\mathbf{X}$. Then, there exists a $\mathbf{G}$-module $\mathbf{V}$, and a closed immersion $\mathbf{X} \hookrightarrow \mathbf{V}$ that is equivariant with respect to the action of $\mathbf{G}$.
\end{prop} 
 
 Recall that a $\mathbf{G}$-module $\mathbf{V}$ is \emph{simple} or \emph{irreducible} if if it has no proper non-zero submodule, and \emph{semi-simple} or \emph{completely reducible} if it is a direct sum of simple submodules.
 From now on, we will assume that the group $\mathbf{G}$ is \emph{reductive}, i.e.\ that its unipotent radical is trivial. We recall the following standard result, that gives other characterizations of complex reductive groups more adapted to our setting.
 
 \begin{prop}[Theorem 1.23 of \cite{brion_introduction_2010}]
 	Let $\mathbf{G}$ be a complex algebraic group. Then the following are equivalent:
 	\begin{enumerate}
 		\item The group $\mathbf{G}$ is reductive.
 		\item The group $\mathbf{G}$ contains no closed normal subgroup isomorphic to the additive group $\CC^n$ for some $n\geq 1$.
 		\item The group $\mathbf{G}$ (viewed as a Lie group) has a compact subgroup $K$ which is dense for the Zariski topology.
 		\item Every finite-dimensional $\mathbf{G}$-module is semi-simple.
 		\item Every $\mathbf{G}$-module is semi-simple.
 		\end{enumerate}
 \end{prop}

\subsection{The complex GIT quotient}
 We fix, once and for all, a reductive algebraic group $\mathbf{G}$, a variety $\mathbf{X}$ and an action of $\mathbf{G}$ on $\mathbf{X}$. When necessary, thanks to \Cref{prop:equivariant_immersion}, we will suppose that $\mathbf{X}$ is a subvariety of some finite dimensional vector space $\mathbf{V}$ and that the action is the restriction of a linear action of $\mathbf{G}$ on $\mathbf{X}$.
 
 We are interested in the space of $\mathbf{G}$-orbits of $\mathbf{X}$. If we consider the usual quotient $\mathbf{X / G}$ as a topological quotient, it is not Hausdorff in general, and has no particular structure.
 The Geometric Invariant Theory (GIT) allows to construct an algebraic variety $\mathbf{X \sslash G}$ and a quotient map $\pi : \mathbf{X} \to \mathbf{X \sslash G}$ satisfying the following properties:
 \begin{itemize}
 	\item The ring of functions $\CC[\mathbf{X \sslash G}]$ is the ring of $\mathbf{G}$-invariant functions of $\mathbf{X}$, that we denote $\CC[\mathbf{X}]^{\mathbf{G}}$.
 	\item The points of $\mathbf{X \sslash G}$ correspond to the closed $\mathbf{G}$-orbits in $\mathbf{X}$ 
 \end{itemize}
 
 Moreover, if $x \in \mathbf{X}$, then there is a unique closed orbit in the closure of $\mathbf{G}x$, and $\pi(x)$ is this precise orbit.
 Thus, $\mathbf{X \sslash G}$ is a quotient in the category $\mathsf{AffAlg}$ of affine algebraic varieties as well as in the category $\mathsf{Htop}$ of Hausdorff topological spaces. 
 
 \begin{rem}
 	If $\mathbf{G}$ acts on a finite dimensional vector space $\mathbf{V}$ and $\mathbf{X} \subset \mathbf{V}$ is a $\mathbf{G}$-invariant closed subset, then we can identify $\mathbf{X \sslash G}$ with the image of $\mathbf{X}$ by the projection map $\pi : \mathbf{V} \to \mathbf{V \sslash G}$.
 \end{rem}
 
\subsection{Real forms}

There are several ways to define real forms. We follow here the definition of Onishchik and Vinberg in \cite{onishchik_lie_1990a} of a real form of a complex vector space or a complex affine variety, that extends to complex algebraic groups.

Given a complex vector space $\mathbf{V}$, a real subspace $V$ is a \emph{real form of $V$} if an $\RR$-basis of $V$ is also a $\CC$-basis of $\mathbf{V}$. In a more abstract frame, $V$ is a real form of $\mathbf{V}$ if the inclusion $V \to \mathbf{V}$ induces an isomorphism $V \otimes_\RR \CC \xrightarrow{\sim} \mathbf{V}$. Observe that a real form of a complex vector space can also be defined by a complex conjugation: then, $V$ will be the set of fixed points of an involutive antiholomorphic automorphism of $\mathbf{V}$. 
This notion extends naturally to affine algebraic varieties. Here, a \emph{real algebraic variety} $X$ will be a subset of some $\RR^n$ defined by polynomial equations; it embeds naturally in a corresponding complex variety $\mathbf{X}$ by extending the scalars. Thus, we have the following definition of a real form.

\begin{defn}\label{def:real_form}
	Let $\mathbf{X}$ be a complex affine variety obtained by a field extension from a real affine variety $X$.
	We say that $X$ is a \emph{real form} of $\mathbf{X}$ if the identity embedding $X \subset \mathbf{X}$ extends to an isomorphism
	$X \otimes_\RR \CC \xrightarrow{\sim} \mathbf{X}$.
	On the level of the rings of functions, $X$ is a real form of $\mathbf{X}$ if and only if $\CC[\mathbf{X}] \simeq \RR[X] \otimes_\RR \CC$.
\end{defn}

A real form $X$ of $\mathbf{X}$ defines a unique involutive antiholomorphic automorphism $x \mapsto \bar{x}$ on $\mathbf{X}$ such that
$
X = \{x \in \mathbf{X} : \bar{x} = x \}
$. We will refer to this automorphism as the \emph{complex conjugation with respect to $X$}.
Conversely, a real form can be defined by an anti-regular involution, as follows:
\begin{prop}[Theorem 6 of {\cite[Chapter 2, §3, 7\textsuperscript{o}]{onishchik_lie_1990a}}]\label{prop:simple_point_real_form}
	Let $\tau$ be an involutive antiholomorphic automorphism\footnote{A more accurate term would be \emph{anti-regular involution}. We keep the term of \cite{onishchik_lie_1990a} in the text, having the same meaning as \emph{anti-regular involution}.}
	of an irreducible complex affine variety $\mathbf{X}$. If the set $X$ of its fixed points contains at least one smooth point, then $X$ is a real form of $\mathbf{X}$.
\end{prop}

This definition extends naturally to complex algebraic groups, so a real algebraic subgroup $G$ is a \emph{real form} of a complex algebraic group $\mathbf{G}$ if the identity embedding $G \subset \mathbf{G}$ extends to an isomorphism $G \otimes_\RR \CC \xrightarrow{\sim} \mathbf{G}$.
In the same way, a real form $G$ of $\mathbf{G}$ defines an involutive antiholomorphic automorphism on $\mathbf{G}$, and such an involution defines a real form if the group $\mathbf{G}$ is irreducible (e.g.\ when it is connected, as the examples that we will consider).
Observe that in all the cases there are several real forms, depending on the involutive antiholomorphic automorphism that is considered. From the viewpoint of complex algebraic varieties and real points, this corresponds to different ways to define a complex variety over $\RR$.
We have, in particular, the definition of twisted forms.

\begin{defn}
	Let $G$ and $G'$ be two real forms of $\mathbf{G}$ with corresponding antiholomorphic involutions $\tau$ and $\tau'$.
	We say that $G$ and $G'$ are \emph{twisted forms} if there is $h \in \mathbf{G}$ such that for all $g \in \mathbf{G}$ we have
	$\tau'(g) = h \, \tau(g) h^{-1}$.
	Note that such an $h$ must satisfy $h \, \tau(h) \in Z(\mathbf{G})$.
\end{defn}

We fix, once and for all, real forms $X$ and $G$ of $\mathbf{X}$ and $\mathbf{G}$ respectively, and we suppose that the corresponding complex conjugations are \emph{compatible}, meaning that for all $g \in \mathbf{G}$ and all $x \in \mathbf{X}$ we have $\conjug{g \cdot x} = \bar{g}\cdot \bar{x}$. Thus, $G$ acts on $X$, and by complexifying this action we recover the action of $\mathbf{G}$ on $\mathbf{X}$.

\subsubsection{Embedding into $\mathbf{G}$-modules and real forms}

 We stated, in \Cref{prop:equivariant_immersion}, that in our setting we can consider $\mathbf{X}$ as a closed $\mathbf{G}$-invariant subset of a finite-dimensional vector space $\mathbf{V}$ with a linear action of $\mathbf{G}$. In this paragraph, we are going to prove that this setting is also valid when considering the real forms $X$ and $G$ of $\mathbf{X}$ and $\mathbf{G}$.
 
 Let $(f_1, \dots , f_d)$ be a finite generating set for the algebra $\RR[X]$. Since $X$ is a real form of $\mathbf{X}$, we know that the extensions of the $f_i$ to $\CC[\mathbf{X}]$, that we still denote by $f_i$, generate the algebra $\CC[\mathbf{X}]$. We are going to embed $\mathbf{X}$ into the dual of a $\mathbf{G}$-invariant subspace of $\CC[X]$.
 Recall that $\mathbf{G}$ acts on $\CC[\mathbf{X}]$ by $g \cdot f : x \mapsto f(g^{-1}x)$. Let $\mathbf{W} = \Span_\CC (\{g\cdot f_i \mid g\in\mathbf{G} , i \in \{1, \dots , d\}\})$; it is a $\mathbf{G}$-invariant subspace of $\CC[\mathbf{X}]$.
 The following lemma follows from the fact that $\CC[\mathbf{X}]$ is a $\mathbf{G}$-module. We refer to Brion's notes \cite{brion_introduction_2010} for details.
 
\begin{lemme}
	The vector space $\mathbf{W} = \Span_\CC (\{g\cdot f_i \mid g \in \mathbf{G} , i \in \{1, \dots , d\}\})$ is finite dimensional.
\end{lemme}

 In order to define the complex conjugation on $\mathbf{W}$, we will use the following lemma.
\begin{lemme}\label{lemma:W_generated_real_orbits}
	We have $\mathbf{W} = \Span_\CC (\{g\cdot f_i \mid g \in G , i \in \{1, \dots , d\}\})$.
\end{lemme}

\begin{proof}
	We only need to prove that every linear form on $\mathbf{W}$ that vanishes at the points $g \cdot f_i$ for $g \in G$ and $i \in \{1,\dots,d\}$ is the zero linear form.
	Let $\varphi \in \mathbf{W}^*$ such that for all $g \in \mathbf{G}$ and all $i \in \{1,\dots ,d\}$ $\varphi(g \cdot f_i) = 0$. 
	For $i\in \{1,\dots,d\}$, let $\psi_i \in \CC[\mathbf{G}]$ be defined by $\psi_i(g) = \varphi(g \cdot f_i)$. Since $\psi_i \vert_G = 0$ and $G$ is a real form of $\mathbf{G}$, we obtain that $\psi_i = 0$. Thus, for all $g \in \mathbf{G}$ and all $i \in \{1,\dots,d\}$, $\varphi(g\cdot f_i) =0$. Hence, $\varphi = 0$.
\end{proof}

 Let $W = \Span_\RR (\{ g \cdot f_i \mid g \in G \, , i \in \{1,\dots,d \} \})$. By \Cref{lemma:W_generated_real_orbits}, we know that $W$ is a real form of $\mathbf{W}$. Let $\mathbf{V}$ be the dual of $\mathbf{W}$ and $V$ the dual of $W$, so $V$ is a real form of $\mathbf{V}$. Observe that $\mathbf{G}$ acts on $\mathbf{V}$ linearly. We are going to embed $\mathbf{X}$ into $\mathbf{V}$ in a $\mathbf{G}$-equivariant way that is also compatible with the complex conjugation. A complete statement that completes the proof is given in the following proposition.

\begin{prop}\label{prop:equiv_embedding_and_real_form}
	Let $\mathbf{G}$ be an algebraic complex reductive group acting on a variety $\mathbf{X}$. Let $G$ and $X$ be compatible real forms of $\mathbf{G}$ and $\mathbf{X}$. Then, there exists a $\mathbf{G}$-module $\mathbf{V}$, a real form $V$ of $\mathbf{V}$ and a closed immersion $\mathbf{X} \hookrightarrow \mathbf{V}$ that is equivariant with respect to the action of $\mathbf{G}$ and the complex conjugation.
\end{prop}

\begin{proof}
	Let $\mathbf{W}$, $W$, $\mathbf{V}$ and $V$ be defined as above. Let
	\[
	 \iota :
	 \begin{aligned}
	 \mathbf{X} &\to \mathbf{V} \\
	 x & \mapsto \ev_x
	 \end{aligned}
	\]
	 where $\ev_x : f \mapsto f(x)$ is the evaluation map. Since $\mathbf{V}$ generates $\CC[\mathbf{X}]$ as a $\CC$-algebra, $\iota$ is a a closed immersion.
	 Now, let $x \in \mathbf{X}$, $g \in \mathbf{G}$ and $f \in \mathbf{W}$. We have
	 \begin{align*}
	 	(g \cdot \iota(x))(f) 
	 	&= (g \cdot \ev_x)(f)
	 	= \ev_x(g^{-1}\cdot f) \\
	 	&= (g^{-1}\cdot f)(x) 
	 	= f(gx) 
	 	= \ev_{gx}(f) \\
	 	&= \iota(gx)(f).
	 \end{align*}
	 Thus, $\iota$ is $\mathbf{G}$-equivariant.
	 
	 It only remains to prove that $\iota$ is equivariant with respect to the complex conjugation.
	 Let $(f_1, \cdots , f_m)$ be an $\RR$-basis for $W$, and $(f_1^*,\dots , f_m^*)$ be its dual basis in $V$. Since $W$ and $V$ are real forms of $\mathbf{W}$ and $\mathbf{V}$, the two families are also bases for the corresponding complex vector spaces.
	 Let $x \in \mathbf{X}$, so $\iota(x) = \sum_{i=1}^{m}\lambda_i f_i^* \in \mathbf{V}$. Let $f = \sum_{i=1}^{m}\mu_i f_i \in \mathbf{W}$.
	 We have
	 \begin{align*}
	 	\iota(\bar{x})(f)
	 	&= \ev_{\bar{x}}(f)
	 	= f(\bar{x})
	 	= \sum_{i=1}^{m}\mu_i f_i(\bar{x}) \\
	 	&= \sum_{i=1}^{m}\mu_i \conjug{f_i(x)} 
	 	= \sum_{i=1}^{m}\mu_i \conjug{\ev_x(f_i)}
	 	= \sum_{i=1}^{m}\mu_i \bar{\lambda_i} \\
	 	&= \sum_{i=1}^{m}\bar{\lambda_i}f_i^*(f) 
	 	= \conjug{\left(\sum_{i=1}^{m}\lambda_i f_i^*\right)}(f)
	 	= \conjug{\ev_x}(f) \\
	 	&=\conjug{\iota(x)}(f)
	 \end{align*}
	 
	 Hence, $\iota(\bar{x}) = \conjug{\iota(x)}$, so $\iota$ is equivariant with respect to the complex conjugation.
\end{proof}

\subsubsection{The variety $\mathbf{X \sslash G}$ is defined over $\RR$}

We are interested in the link between the quotients of $X$ by $G$ and the one of $\mathbf{X}$ by $\mathbf{G}$. First, we observe that the GIT quotient $\mathbf{X \sslash G}$ is defined over $\RR$, so it makes sense to consider its real points. We first need a technical lemma. The proof is inspired by the proof of Proposition 6.7 of \cite{richardson_minimum_1990}.

\begin{lemme}\label{lemma:RXG_finitely_generated}
	Let $X$ be a real form of $\mathbf{X}$. Then, the ring of invariant functions $\RR[X]^G$ is finitely generated.
\end{lemme}
\begin{proof}
	Since $G$ is Zariski-dense in $\mathbf{G}$ and $X$ is a real form of $\mathbf{X}$, we have
	\[
	\CC[\mathbf{X}]^{\mathbf{G}}
	=
	\CC[\mathbf{X}]^G
	=
	(\RR[X] \otimes_\RR \CC)^G
	=
	\RR[X]^G \otimes_\RR \CC .
	\]
	
	By Nagata's theorem (see \cite{nagata_invariants_1963} or \cite[Theorem 1.24]{brion_introduction_2010}),
	we know that $\CC[\mathbf{X}]^{\mathbf{G}}$ is finitely generated. Let $(f_1, \cdots , f_d)$ be a finite set of generators. Since $\CC[\mathbf{X}]^{\mathbf{G}} = \RR[X]^G \otimes_\RR \CC$, for $j \in \{1, \dots , d \}$ there exist $g_j,h_j \in \RR[X]^G$ such that $f_j \vert_X = g_j + i h_j$. Let $R \subset \RR[X]$ be the ring generated by $\{ g_1, \dots , g_d , h_1,\dots , h_d \}$. Since $f_j \vert_X = g_j + i h_j$, we have that $f_1, \cdots , f_d \in R \otimes_\RR \CC$. Thus, $R \otimes_\RR \CC = \CC[\mathbf{X}]^{\mathbf{G}}$ and $R = \RR[X]^{G}$ is generated by $\{ g_1, \dots , g_d , h_1,\dots , h_d \}$.
\end{proof}

\begin{prop}\label{prop:complex_conjug_compatible_GIT}
	The complex conjugations on $\mathbf{G}$ and $\mathbf{X}$ induce a compatible complex conjugation in $\mathbf{X \sslash G}$.
\end{prop}
\begin{proof}
	By \Cref{lemma:RXG_finitely_generated}, we can choose a finite generating set $(f_1, \dots , f_d)$ for $\RR[X]^G$.
	Since $\CC[\mathbf{X}]^{\mathbf{G}} = \RR[X]^G \otimes_\RR \CC$, the extensions of the $f_i$ to $\CC[\mathbf{X}]$ generate $\CC[\mathbf{X}]^{\mathbf{G}}$.
	Therefore, we can identify $\mathbf{X \sslash G}$ with the image of the map
	\[
	\begin{aligned}
	\mathbf{X} &\to \CC^d \\
	x &\mapsto (f_1(x), \dots , f_d(x))
	\end{aligned}
	\]
	Hence, the complex conjugation of $\CC^d$ defines an antiholomorphic involutive automorphism $\tau$ of $\mathbf{X}$ such that
	\[
	\forall x \in \mathbf{X} \quad
	\pi(\bar{x}) = \tau(\pi(x))
	\]
	Observe that this last condition shows that the definition of $\tau$ on $\mathbf{X \sslash G}$ is independent of the choice of $f_1, \dots , f_d$.
\end{proof}

\begin{rem}
	This complex conjugation allows to define $\mathbf{X \sslash G}$ over $\RR$, and to consider its real points $(\mathbf{X \sslash G})(\RR)$ as the fixed points of $\tau$. However, even if $\mathbf{X}$ is irreducible, $\mathbf{X \sslash G}$ is not necessarily irreducible and there might not be a real smooth point. Hence, the statement above does not mean that $(\mathbf{X \sslash G})(\RR)$ is a real form of $(\mathbf{X \sslash G})$.
\end{rem}

\subsection{A real GIT quotient}

 We consider now the action of $G$ on $X$, instead of the one of $\mathbf{G}$ on $\mathbf{X}$. In this setting, the ring of invariant functions $\RR[X]^G$ does no longer define a real algebraic quotient, but we can still consider the space $X \sslash G$ of closed $G$-orbits.
 
 We summarize here the results obtained by Richardson and Slodowy in \cite{richardson_minimum_1990} about the space $X \sslash G$ and a natural map $\beta : X \sslash G \to \mathbf{X \sslash G}$.
 The reader might be aware that the results of \cite{richardson_minimum_1990} are stated for actions of $G$ on a vector space $V$. However, by \Cref{prop:equiv_embedding_and_real_form}, we can embed $\mathbf{X}$ as a closed $\mathbf{G}$-invariant subset of a complex vector space $\mathbf{V}$ endowed with a linear action of $\mathbf{G}$, in such a way that the actions of $\mathbf{G}$ and the complex conjugation are equivariant. Hence, $X$ will be embedded as a closed $G$-invariant subset of $V$. Since the results are about closed orbits, they generalize without any trouble to the action of $G$ on a real form $X$ of an algebraic variety $\mathbf{X}$. The situation is summarized in the diagram of \Cref{fig:diagram_quotients_complete}.

 \begin{figure}[htbp]
 	\[
 	\begin{tikzcd}[row sep=1.5em, column sep = 1.5em]
 	X \arrow[rr, hook] \arrow[dr, hook, "\iota" description] \arrow[dd, two heads, "p" description] 
 	&&
 	\mathbf{X} \arrow[dd, two heads, "\pi" description] \arrow[dr,hook,"\iota" description] \\
 	& V \arrow[rr, hook, crossing over]
 	&&
 	\mathbf{V} \arrow[dd, two heads,"\pi" description] \\
 	X \sslash G \arrow[rr, "\qquad  \beta" description] \arrow[dr, hook] 
 	&& \mathbf{X \sslash G} \arrow[dr, hook] \\
 	& V \sslash G \arrow[rr, "\beta" description] \arrow[from=uu, crossing over, two heads, "p" description] 
 	&& \mathbf{V \sslash G}
 	\end{tikzcd}
 	\]
 	\caption{Maps between $\mathbf{X}$, $\mathbf{V}$, their real forms and their quotients.} \label{fig:diagram_quotients_complete}
 \end{figure}

 Let $X \sslash G$ be the space of closed $G$-orbits of $X$, endowed with the quotient topology.
 Luna, in \cite{luna_certaines_1975},
 and 
 Richardson and Slodowy, in \cite{richardson_minimum_1990},
 prove that $X \sslash G$ is the Hausdorffization of the topological quotient $X / G$, so it is the quotient of $X$ by $G$ in the category of Hausdorff spaces. 
 The techniques used by Richardson and Slodowy involve a detailed study of the set of minimum vectors in $V$
 for a norm invariant by the maximal compact subgroup of $G$.
 With these tools, they are able to prove, in \cite[Theorem 7.7]{richardson_minimum_1990}, that $X \sslash G$ is homeomorphic to a semi-algebraic set in some $\RR^d$.

 Let $p : X \to X \sslash G$ be the projection map (noted $\pi$ in \cite{richardson_minimum_1990}). It is a continuous map, that is the quotient map in the category of Hausdorff spaces.
 Keeping this time the notation of \cite[Section 7]{richardson_minimum_1990}, since $\pi \vert_{X} : X \to \mathbf{X \sslash G}$ is continuous, $G$-invariant, and has as target a Hausdorff space, there is a natural continuous map $\beta : X \sslash G \to \mathbf{X \sslash G}$ such that $\beta \circ p = \pi \vert_{X}$.
 Richardson and Slodowy prove that the map $\beta$ is proper (Proposition 7.4) and has finite fibers (Lemma 8.2).
 Moreover, if $E$ is a closed $G$-invariant subset of $X$, then $p(E)$ is closed in $X \sslash G$, and hence $\pi(E)$ is closed in $\mathbf{X\sslash G}$.
 The same closure properties hold for arbitrary local fields of characteristic zero, as proven by Bremigan in \cite{bremigan_quotients_1994}, and the restrictions of $\beta$ to the strata defined in the same article are covering maps.

 We are interested in the relations between the quotients $X \sslash G$ and $\mathbf{X \sslash G}$. 
 By \Cref{prop:complex_conjug_compatible_GIT}, we know that the complex conjugations of $\mathbf{X}$ and $\mathbf{G}$ induced by $X$ and $G$ are compatible with a complex conjugation of $\mathbf{X \sslash G}$, so $\pi(X) = \beta(X \sslash G)$ is a subset of $(\mathbf{X \sslash G})(\RR)$.
 In addition, observe that $X$ is a real algebraic set, so, by the Tarski-Seidenberg theorem, its image by a real algebraic map is a semi-algebraic set. In particular, $\pi(X)$ is a semi-algebraic subset of $(\mathbf{X \sslash G})(\RR)$. 
 A summary of the setting, without the ambient vector spaces, is given in \Cref{fig:diagram_quotients_X}.
 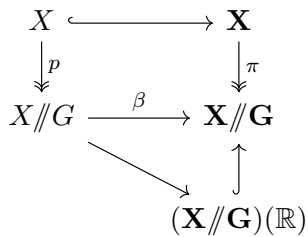
\begin{figure}[htbp]
 	\[ 
 	\begin{tikzcd}
 	X \arrow[r, hook] \arrow[d, two heads,"p"] 
 	&
 	\mathbf{X} \arrow[d, two heads, "\pi"] \\
 	X \sslash G \arrow[r, "\beta"] \arrow[dr]
 	&
 	\mathbf{X \sslash G} \\
 	&
 	(\mathbf{X \sslash G})(\RR) \arrow[u, hook]
 	\end{tikzcd}
 	\]
 	\caption{Maps between $X$, $\mathbf{X}$ and their quotients.}
 	\label{fig:diagram_quotients_X}
 \end{figure}

\section{Lifting real points} \label{section:lifting_real_points}
 In this section, we are interested in the relations between the real GIT quotient $X \sslash G$ and the real points of the complex GIT quotient $\mathbf{X \sslash G}$.
 First, we recall briefly the setting given in \Cref{sect:background_and_setting}. Let 
 $\mathbf{G}$ be a connected complex reductive algebraic group
 acting on a complex variety $\mathbf{X}$.
 Let $G$ and $X$ be real forms of $\mathbf{G}$ and $\mathbf{X}$ respectively, that are compatible with the $\mathbf{G}$-action.
 Let $\mathbf{X \sslash G}$ be the GIT quotient of $\mathbf{X}$ by $\mathbf{G}$: it is an algebraic set defined over $\RR$ whose ring of invariants is $\CC[\mathbf{X}]^{\mathbf{G}}$.
 Let $X \sslash G$ be the real GIT quotient of $X$ by $G$, that is a topological Hausdorff space.
 
 We have also the projection maps $\pi: \mathbf{X} \to \mathbf{X \sslash G}$ and $p : X \to X \sslash G$, as well as the natural map $\beta : X \sslash G \to \mathbf{X \sslash G}$, as in \Cref{fig:diagram_quotients_X}.
 As we will see later in the example of \Cref{section:example}, the map $\beta$ may not be surjective nor injective. However, under some extra hypotheses, we can still have some information on the surjectivity and injectivity of the map.
 We begin by considering in \Cref{subsect:compact_real_form} the case when $G$ is the compact real form of $\mathbf{G}$, and where $\beta$ is a homeomorphism onto its image. Then, in \Cref{subsect:lifting_irreducible_points}, we prove that we can lift neat points of $(\mathbf{X \sslash G})(\RR)$ to quotients of the form $X' \sslash G'$, but maybe after changing the real forms $X$ and $G$. 
 Finally, in \Cref{subsect:number_of_lifts}, we construct a bijection between the lifts of a point of $(\mathbf{X \sslash G})(\RR)$ to $X \sslash G$ and a subset of some cohomology set.

\subsection{The compact real form}\label{subsect:compact_real_form}
 We suppose here that $G = K$ is the compact real form of $\mathbf{G}$.
 Then, all the $G$-orbits of $X$ are compact, and the quotients $X/G$ and $X \sslash G$ are the same.
 We are going to prove that, in this case, the quotient $X \sslash G$ is homeomorphic to its image by the map $\beta : X \sslash G \to \mathbf{X \sslash G}$. We will use the following lemma, proven by Birkes in \cite{birkes_orbits_1971} for a general real form $G$. In the original paper it is stated for the ambient vector spaces $\mathbf{V}$ and a real form $V$, but the result generalizes immediately to our setting.

\begin{lemme}[Corollary 5.3 of \cite{birkes_orbits_1971}]\label{lemma:corollary_birkes_closed_orbits}
	Let $x \in X$ be such that $G \cdot x$ is closed in $X$. Then, $\mathbf{G} \cdot x$ is closed in $\mathbf{X}$.
\end{lemme}

\begin{prop}\label{prop:compact_real_form}
	If $G=K$ is the compact form of $\mathbf{G}$, then the natural map $X/K \to \mathbf{X \sslash G}$ is a homeomorphism onto its image.
\end{prop}
\begin{proof}
	Since $K$ is compact, all the $K$-orbits of $X$ are compact, and so the topological quotient $X/K$ is Hausdorff. Hence, the projection $X/K \to X \sslash K$ is a homeomorphism.
	By an abuse of language, we will identify these two spaces, and let $\beta : X/K \to \mathbf{X \sslash G}$ be the natural map, which is a continuous map between locally compact Hausdorff spaces.  
	
	By \cite[Proposition 7.4]{richardson_minimum_1990}, we know that $\beta$ is proper. Thus, in order to prove the statement it only remains to prove that $\beta$ is injective.
	
	Let $x_1,x_2 \in X$ such that $\beta (K x_1) = \beta (K x_2)$.
	Since $K x_1$ and $K x_2$ are compact, they are closed in $X$, and hence, by \Cref{lemma:corollary_birkes_closed_orbits}, $\mathbf{G} x_1$ and $\mathbf{G} x_2$ are closed in $\mathbf{X}$.
	Thus, for $i\in \{1,2\}$, $\beta(K x_i)$ is the closed orbit $\mathbf{G} x_i$ and hence $x_2 \in \mathbf{G}x_1$.
	By \cite[Proposition 8.3.1]{richardson_minimum_1990}, $\mathbf{G}x_1 \cap X = Kx_1$, so $x_2 \in Kx_1$ and $Kx_1 = K x_2$, which proves the injectivity of $\beta$.
	
\end{proof}

\begin{rem}
	The statement of \Cref{prop:compact_real_form} was already proven by Florentino and Lawton in \cite[Proposition 4.5]{florentino_character_2013} in the setting of quiver representations, that contains the case of character varieties that we discuss later. The main tool in their work for proving injectivity is the polar decomposition.
\end{rem}

\subsection{Lifting real neat points}\label{subsect:lifting_irreducible_points}

We know that the image of the map $\beta : X \sslash G \to \mathbf{X \sslash G}$ is contained in the set of real points of $\mathbf{X \sslash G}$.
Conversely, we would like to lift points of $(\mathbf{X \sslash G})(\RR)$ to points of $X \sslash G$. 
We are going to prove that we can do it under two conditions. First, we need to assume that the point is neat and, second, we need to allow to recover different real forms $X'$ of $\mathbf{X}$ and $G'$ of $\mathbf{G}$.
We also make the additional assumption that $Z(\mathbf{G})$ acts trivially on $\mathbf{X}$.

\begin{defn}\label{def:irreducible_point_X}
	Consider an action of $\mathbf{G}$ on $\mathbf{X}$. We say that a point $x \in \mathbf{X}$ is \emph{neat} if the three following conditions hold.
	\begin{enumerate}
		\item The point $x \in \mathbf{X}$ is regular.
		\item The orbit $\mathbf{G} \cdot x$ is closed in $\mathbf{X}$. \label{cond_def:irred_closed_orbit}
		\item The stabilizer of $x$ in $\mathbf{G}$ is $Z(\mathbf{G})$.\label{cond_def:irred_central_stabilizer}
		
	\end{enumerate}
\end{defn}

\begin{rem}
	If the center of $\mathbf{G}$ is finite, then a neat point is also a stable point, in the sense of \cite[Definition 1.25]{brion_introduction_2010}.
\end{rem}

\begin{defn}\label{def:irred_point_quotient}
	Let $x \in \mathbf{X \sslash G}$. We say that $x$ is neat if there is a neat point $x \in \mathbf{X}$ such that $x = \pi(x)$.
\end{defn}

The following proposition states that we can lift real neat points of $\mathbf{X \sslash G}$ to real forms of the component of $x$, stable by an action of a real form of $\mathbf{G}$.
A similar proof of this statement can be found in \cite[Page 68]{johnson_deformation_1987} in the particular case of a character variety.

\begin{thm}
	 \label{thm:lift_real_irred}
	Let $x \in \mathbf{X}$ be a neat point such that $\pi(x) \in (\mathbf{X \sslash G})(\RR)$.
	Let $\mathbf{Y} \subset \mathbf{X}$ be the irreducible component of $x$.
	Then, there exist real forms $G'$ of $\mathbf{G}$ and $Y'$ of $\mathbf{Y}$ such that
	$G'$ is a twisted form of $G$,
	$x \in Y'$ and $Y'$ is $G'$-stable.
\end{thm}

\begin{proof}
	Since $\pi(x) = \conjug{\pi(x)}$, we have
	$
	\pi(\bar{x}) = \conjug{\pi(x)} = \pi(x)$.
	Thus, there exists $h \in \mathbf{G}$ such that $x = h \cdot \bar{x}$. We have
	\begin{align*}
	\bar{x} &= \bar{h} \cdot x \\
	h \cdot \bar{x} &= h\bar{h} \cdot x \\
	x &= h\bar{h} \cdot x
	\end{align*}
	
	Since $x$ is neat, $h \bar{h} \in Z(\mathbf{G})$.
	Now, define
	\begin{align*}
	\theta \colon &
	\begin{aligned}
	\mathbf{G} &\to \mathbf{G} \\
	g &\to h \bar{g} h^{-1} 
	\end{aligned}
	&
	\sigma \colon &
	\begin{aligned}
	\mathbf{X} &\to \mathbf{X} \\
	u &\to h \cdot \bar{u}
	\end{aligned}
	\end{align*}
	
	We claim that $\theta$ and $\sigma$ are compatible involutive antiholomorphic automorphisms. Indeed, since $h\bar{h} \in Z(\mathbf{G})$ and $Z(\mathbf{G})$ acts trivially on $\mathbf{X}$, we have, for $g \in \mathbf{G}$ and $u \in \mathbf{X}$
	\begin{align*}
	\theta \circ \theta (g) &= \theta(h \bar{g} h^{-1}) 
	&
	\sigma \circ \sigma (u) &= \sigma(h \cdot \bar{u}) \\
	&= (h\bar{h}) g (h\bar{h})^{-1} 
	&
	&= h \bar{h} \cdot u \\
	&= g 
	&
	&= u
	\end{align*}
	Furthermore, $\theta(g) \cdot \sigma(u) = h\bar{g}h^{-1} \cdot (h \cdot \bar{u}) = h \cdot \conjug{(g\cdot u)} = \sigma(g\cdot u)$.
	
	Let $G' \subset \mathbf{G}$ and $X' \subset \mathbf{X}$ be the fixed points of $\theta$ and $\sigma$ respectively, and let $Y' = \mathbf{Y} \cap X'$.
	Since $x \in \mathbf{Y}$ is fixed by $\sigma$, we have $x \in Y'$.
	Moreover, since $\theta$ and $\sigma$ are compatible, $X'$ and $Y'$ are $G'$-stables.
	By the definition of $G'$, we know that it is a twisted\ form of $G$.
	It only remains to check that $Y'$ is a real form of $\mathbf{Y}$, which follows from \Cref{prop:simple_point_real_form}, since $\mathbf{Y}$ is irreducible and $x \in Y'$ is a regular point.
\end{proof}

\begin{rem}
	On the one hand, if $\mathbf{X} = \mathbf{V}$ is a vector space, then it is irreducible, so $\mathbf{Y} = \mathbf{V}$ and $Y'$ is a real form of $\mathbf{V}$.
	On the other hand, if $\mathbf{X}$ is $\mathbf{G}$-equivariantly embedded in a vector space $\mathbf{V}$, \Cref{thm:lift_real_irred} gives real forms $V'$ of $\mathbf{V}$ and $Y'$ of $\mathbf{Y}$ that are stable by $G'$.
	Thus, we can consider the set $X'$ of fixed points of $\sigma$ in $\mathbf{X}$, that is $G'$-invariant and satisfies $Y' \subset X' \subset V'$. However, $X'$ might not be a real form of $\mathbf{X}$ in the sense of \Cref{def:real_form}.
\end{rem}

\begin{rem}
	If we consider the rings of functions instead of the points of the varieties, then the induced action of $\sigma$ on $\CC[\mathbf{X}]$ is given by $f \mapsto \sigma(f)$, where
	$\sigma(f) : x \mapsto \conjug{f(\sigma(x))}$.
	Let $\CC[\mathbf{X}]^{\sigma}$ be the set of fixed points of $\sigma$ for this action. Then, $\CC[\mathbf{X}]^{\sigma}$ is an $\RR$-algebra such that $\CC[\mathbf{X}]^{\sigma} \otimes_\RR \CC = \CC[\mathbf{X}]$.
\end{rem}

\subsection{The number of lifts}\label{subsect:number_of_lifts}

By the work of Richardson and Slodowy in \cite{richardson_minimum_1990}, we know that the projection $\beta : X \sslash G \to \mathbf{X \sslash G}$ has finite fibers. In this section, we will relate the fiber of a point with some cohomological invariant of a stabilizer.
The identification is mainly a consequence of Corollary 1 of \cite[Ch.\ I §5.4 prop.\ 36]{serre_cohomologie_1994d}, but we keep the full construction for completeness.

\subsubsection{The first non-Abelian cohomology set}
We keep our setting, with $G$ a real form of $\mathbf{G}$. The complex conjugation induced by $G$ is precisely an action of the Galois group $\mathrm{Gal}(\CC/\RR)$ on $\mathbf{G}$.
A \emph{cocycle} for this action is a function $f: \mathrm{Gal}(\CC/\RR) \to \mathbf{G}$ such that for all 
$\sigma, \tau \in \mathrm{Gal}(\CC/\RR)$ we have 
$f(\sigma \tau) = f(\sigma) \: \sigma(f(\tau))$.
Observe that this is equivalent to the choice of the image $g \in \mathbf{G}$ of the non-trivial element of $\mathrm{Gal}(\CC/\RR)$, that must satisfy $g\bar{g}= \mathrm{id}$.
A \emph{coboundary} for the action is a map of the form $\sigma \mapsto h^{-1}\sigma(h)$ for some $h \in \mathbf{G}$; this corresponds to the elements $g \in \mathbf{G}$ that can be written as $h^{-1} \bar{h}$.
We can define the first cohomology set (that may not have a group structure if $\mathbf{G}$ is not Abelian) by mimicking the idea of a quotient of cocycles by coboundaries. We follow Section 5.1 of Chapter I of Serre's book \cite{serre_cohomologie_1994d} for the following definition.

\begin{defn}
	Let $\tau$ be an involutive antiholomorphic automorphism of $\mathbf{G}$ defining a real form $G$.
	The first cohomology set of $\mathrm{Gal}(\CC/\RR)$ is the set 
	\[
	H^1_\tau(\mathrm{Gal}(\CC/\RR),\mathbf{G}) =
	\{g \in \mathbf{G} \mid g \, \tau(g) = \mathrm{id} \} / \sim \] where
	$g_1 \sim g_2$ if there exists $h\in \mathbf{G}$ such that $g_1 = h^{-1}g_2 \, \tau(h)$.
\end{defn}

\begin{rem}
	This definition depends on the real form $G$ and the involution $\tau$. For example, on the one hand, the generalization of Hilbert's Theorem 90 given in \cite[X.1, corollay of Prop.3]{serre_cohomologie_1994d} states that $H^1(\mathrm{Gal}(\CC/\RR) , \mathrm{SL}_n(\CC) )$ is trivial for the usual complex conjugation of the real form $\mathrm{SL}_n(\RR)$. On the other hand, if $\tau : M \mapsto {}^t\!\bar{M}^{-1}$ is the involution defining $\mathrm{SU}(n)$, one can show that there is a one-to-one correspondence between $H^1_\tau(\mathrm{Gal}(\CC/\RR),\mathrm{SL}_n(\CC))$ and the set of signatures of type $(p,2q)$ of non-degenerate Hermitian forms, and hence contains $\lfloor \frac{n}{2} \rfloor +1$ elements.
	
	In the rest of the section, we will write $\bar{g}$ for $\tau(g)$ and not write the subindex $\tau$ in the $H^1$ to avoid heavy notation. However, the reader should keep in mind that the objects depend on the real forms and complex conjugations chosen.
\end{rem}

\subsubsection{A map from orbits to cohomology sets}
Let $x \in X$ be such that $Gx$ is closed in $X$ (and hence, by \Cref{lemma:corollary_birkes_closed_orbits}, $\mathbf{G}x$ is also closed in $\mathbf{X}$). Let $\mathbf{G}_x$ be the stabilizer of $x$ in $\mathbf{G}$. Observe that the cardinal of the fiber $\beta^{-1}(\beta(p(x))) \subset X \sslash G$ is the number of $G$-orbits for the action of $G$ on $\mathbf{G}x \cap X$.
Now, we construct a map from this set of orbits to $H^1(\mathrm{Gal}(\CC/\RR), \mathbf{G}_x)$.
First, define the map 
\[
\tilde{\varphi} \colon
\begin{aligned}
\mathbf{G}x \cap X &\to H^1(\mathrm{Gal}(\CC/\RR), \mathbf{G}_x) \\
g \cdot x &\mapsto [g^{-1}\bar{g}]
\end{aligned}
\]

\begin{lemme}
	The map $\tilde{\varphi}$ is well defined.
\end{lemme}
\begin{proof}
	Let $g \in \mathbf{G}$ such that $g \cdot x \in X$. 
	Hence, $\bar{g}\cdot x = \bar{g}\cdot \bar{x} =\conjug{g \cdot x} = g \cdot x$, so $g^{-1}\bar{g} \in \mathbf{G}_x$. Moreover, $g^{-1}\bar{g}$ is a cocycle, so the class $[g^{-1}\bar{g}]$ is a well defined element in $H^1(\mathrm{Gal}(\CC/\RR), \mathbf{G}_x)$.
	It remains to check that if $h \in \mathbf{G}$ is such that $h \cdot x = g \cdot x$, then $[h^{-1}\bar{h}] = [g^{-1}\bar{g}]$. If $h$ is such an element, then $s = g^{-1}h \in \mathbf{G}_x$, and $h^{-1}\bar{h} = s^{-1}g^{-1}\bar{g}\bar{s}$, so $[h^{-1}\bar{h}] = [g^{-1}\bar{g}]$ and $\tilde{\varphi}$ is well defined.
	
\end{proof}
\begin{lemme}
	The map $\tilde{\varphi}$ is constant on $G$-orbits.
\end{lemme}
\begin{proof}
	Let $g \in \mathbf{G}$ such that $g \cdot x \in X$, and let $h \in G$. Then
	$(hg)^{-1} \conjug{hg} = g^{-1}h^{-1}\bar{h}\bar{g} = g^{-1}\bar{g}$, so $\tilde{\varphi}(h\cdot (g\cdot x)) = \tilde{\varphi}(g \cdot x)$.
\end{proof}

Thus, the map $\widetilde{\varphi}$ factors through the space of orbits to a map 
\[\varphi : (\mathbf{G}x \cap X)/G \to H^1(\mathrm{Gal}(\CC/\RR), \mathbf{G}_x).\]

\begin{prop}\label{prop:phi_injective}
	The map $\varphi$ is injective.
\end{prop}
\begin{proof}
	Let $g,h \in \mathbf{G}$ such that $g \cdot x \in X$, $h \cdot x \in X$ and $\tilde{\varphi}(g \cdot x) = \tilde{\varphi}(h \cdot x)$.
	Hence, $[h^{-1}\bar{h}] = [g^{-1}\bar{g}]$. So, maybe after multiplying $h$ by an element of $\mathbf{G}_x$, we can suppose that $h^{-1}\bar{h} = g^{-1}\bar{g}$.
	Thus, $\bar{h}\bar{g}^{-1} = hg^{-1} \in G$, and therefore $g \cdot x$ and $h \cdot x = (hg^{-1})\cdot (g \cdot x)$ are in the same $G$-orbit.
\end{proof}

Since the stabilizer of a neat point is the center of $\mathbf{G}$, we obtain the following corollary.

\begin{cor}\label{cor:beta_one_to_one_if_h1_trivial}
	Suppose that $H^1(\mathrm{Gal}(\CC/\RR),Z(\mathbf{G}))$ is trivial. Then the restriction of the map $\beta : X\sslash G \to \mathbf{X \sslash G}$ to the neat points is injective.
\end{cor}
\begin{rem}
	The condition of \Cref{cor:beta_one_to_one_if_h1_trivial} holds for the groups $\mathrm{GL}_n(\CC)$, $\mathrm{SL}_{2m+1}(\CC)$ and $\mathrm{SO}_{2m+1}(\CC)$ for the usual complex conjugation. It also holds for $\mathrm{SL}_{2m+1}(\CC)$ for the twisted real forms of $\mathrm{SU}(2m+1)$.
\end{rem}

	Observe that there is a natural map of pointed sets from $H^1(\galCR ,\mathbf{G}_{x})$ to $H^1(\galCR ,\mathbf{G})$. Indeed, the condition of being a cocycle does not depend on the ambient group, and if $g \sim g'$ in $\mathbf{G}_{x}$, then $g \sim g'$ also in $\mathbf{G}$.
	The map $\varphi$ induces the one-to-one correspondence described in the following proposition.
\begin{prop}\label{prop:bound_number_of_lifts}
	Let $x \in (\mathbf{X \sslash G})(\RR)$, $\tilde{x} \in \mathbf{X}$ be a lift of $x$ and $\mathbf{G}_{\tilde{x}}$ the stabilizer of $\tilde{x}$ in $\mathbf{G}$. Then, 
	there is a one-to-one correspondence between $\beta^{-1}(x)$ and the kernel of the natural map
	\[ 
	H^1(\galCR ,\mathbf{G}_{\tilde{x}} ) \to H^1(\galCR ,\mathbf{G} ).
	\]
\end{prop}
\begin{proof}
	By \Cref{prop:phi_injective}, we know that there is a one-to-one correspondence between $\beta^{-1}(x)$ and the image of $\varphi$ in $H^1(\galCR , \mathbf{G}_{\tilde{x}})$. Thus, we only need to prove that the image of $\varphi$ is precisely 
	$\ker(H^1(\galCR ,\mathbf{G}_{\tilde{x}} ) \to H^1(\galCR ,\mathbf{G}))$.\\
	On the one hand, an element of the image of $\varphi$ is of the form $[g^{-1}\bar{g}]$ for some $g \in \mathbf{G}$, so its image in $H^1(\galCR ,\mathbf{G})$ is trivial. \\
	On the other hand, let $[h] \in \ker(H^1(\galCR ,\mathbf{G}_{\tilde{x}} ) \to H^1(\galCR ,\mathbf{G}))$. 
	Thus, $h \in \mathbf{G}_{\tilde{x}}$ and there is $g \in \mathbf{G}$ such that $h = g^{-1} \bar{g}$.
	Let $y = g \tilde{x} \in \mathbf{X}$. We have 
	\[ \bar{y} = \overline{g \tilde{x}} = \bar{g}\tilde{x} = \bar{g}(h^{-1} \tilde{x}) = g \tilde{x} = y \]
	so $y = \bar{y} \in X$. Hence, $\tilde{\varphi}(y) = [g^{-1} \bar{g}] = [h] \in H^1(\galCR ,\mathbf{G}_{\tilde{x}} )$, and $[h] \in \Im{\varphi}$.
\end{proof}

\begin{cor}
	If $H^1(\mathrm{Gal}(\CC / \RR) , \mathbf{G})$ is trivial, then $\varphi$ is surjective.
\end{cor}

\begin{rem}
	By \cite[X.1, corollay of Prop. 3]{serre_cohomologie_1994d}, if $\mathbf{G}$ is $\mathrm{GL}_n(\CC)$ or $\mathrm{SL}_n(\CC)$, and $G$ is $\mathrm{GL}_n(\RR)$ or $\mathrm{SL}_n(\RR)$, then the cohomology set $H^1(\mathrm{Gal}(\CC/\RR),\mathbf{G})$ is trivial. In that case, $\varphi$ induces a one-to-one correspondence between the fiber $\beta^{-1}(\beta(p(x)))$ and $H^1(\mathrm{Gal}(\CC/\RR), \mathbf{G}_x)$.
\end{rem}

\section{Application to character varieties}\label{section:char-var}

 A broadly studied class of complex GIT quotients that are useful for studying finitely generated groups and moduli spaces of geometric structures are \emph{character varieties}. 
 For a finitely generated group $\Gamma$ and a complex reductive group $\mathbf{G}$, they are defined as $\mathfrak{X}_{\mathbf{G}}(\Gamma) = \Hom(\Gamma, \mathbf{G}) \mathbf{ \sslash \mathbf{G}}$, where $\mathbf{G}$ acts by conjugation.
 We recall below the general setting and definition, but we refer to the survey of Sikora \cite{sikora_character_2012a} and the references within for a more complete exposition.
 
 In our general notation, we will have $\mathbf{X} = \Hom(\Gamma , \mathbf{G})$, and given a real form $G$ of $\mathbf{G}$, the corresponding real form will be then $X = \Hom(\Gamma , G)$.
 We recall below some known facts about character varieties when $\mathbf{G}$ is a classical complex group. Note that in most cases, the trace functions generate the ring of invariant functions, so it is not difficult to compute the complex conjugation of $\mathbf{X \sslash G}$, even in explicit examples.
 
 In \cite[Proposition 6.1]{casimiro_topology_2014}, Casimiro, Florentino, Lawton and Oliveira observe that te points of $\Hom(\Gamma , G)$ are projected to the real points $\mathfrak{X}_\mathbf{G}(\Gamma)(\RR)$. We deal here with the converse statement of lifting real characters to representations with values in real forms.
 Lifting real points of the a character variety $\mathfrak{X}_\mathbf{G}(\Gamma)$ to $\Hom(\Gamma, G)$ had already been considered by Johnson and Millson in \cite{johnson_deformation_1987} for the groups $G = \mathrm{SO}(n,1)$. 
 The author has also studied lifting of real points of $\mathrm{SL}_n(\CC)$-character varieties in \cite{acosta_character_2019a}, while carefully describing the example $\Gamma = \ZZ/3\ZZ * \ZZ / 3 \ZZ$, and has proven the lifting property, by considering independently each case, for the complex groups $\mathrm{GL}_n(\CC)$, $\mathrm{SL}_n(\CC)$, $\mathrm{O}_n(\CC)$, $\mathrm{SO}_n(\CC)$ or $\mathrm{Sp}_{2n}(\CC)$ in 
 \cite{acosta_character_2019}. \Cref{thm:lift_real_characters} generalizes all the results.

\subsection{Character varieties as GIT quotients}

 Let $\Gamma$ be a finitely generated group and $\mathbf{G}$ be a complex reductive algebraic group. If $(\gamma_1, \dots , \gamma_s)$ is a set of generators for $\Gamma$, the set $\Hom(\Gamma,\mathbf{G})$ of representations of $\Gamma$ into $\mathbf{G}$ can be identified with a subset of $\mathbf{G}^s$ via the injective map
 \[
 \begin{aligned}
 \Hom(\Gamma , \mathbf{G}) &\to \mathbf{G}^s \\
 \rho &\mapsto (\rho(\gamma_1), \dots \rho(\gamma_s)).
 \end{aligned}
 \]
 The image of this map is the subset of $\mathbf{G}^s$ for which the elements satisfy the relations of $\Gamma$. In this way, $\Hom(\Gamma,\mathbf{G})$ is given by polynomial equations and is a subvariety of $\mathbf{G}^s$.
 The group $\mathbf{G}$ acts algebraically on $\mathbf{G}^s$ by simultaneous conjugation, and the subvariety $\Hom(\Gamma,\mathbf{G})$ is stable for this action. Thus, we can consider the \emph{$\mathbf{G}$-character variety} of $\Gamma$, that is
 the GIT quotient \[\mathfrak{X}_\mathbf{G}(\Gamma) = \Hom(\Gamma , \mathbf{G}) \sslash \mathbf{G}.\]

 Recall that a subgroup $\mathbf{H}$ of $\mathbf{G}$ is \emph{irreducible} it is not contained in any proper parabolic subgroup of $\mathbf{G}$, and \emph{completely reducible} if for every parabolic subgroup $\mathbf{P}$ of $\mathbf{G}$ containing $\mathbf{H}$ there is a Levi subgroup $\mathbf{L}$ of $\mathbf{P}$ such that $\mathbf{H \subset L}$.
 Then, as stated in \cite[Theorem 30]{sikora_character_2012a}, the closed orbits of this action are precisely the ones of \emph{completely reducible representations}, meaning that their image in $\mathbf{G}$ is a completely reducible subgroup. 

 Observe that if $\mathbf{G}$ is a linear group, the trace functions, defined as
 $\rho \mapsto \tr(\rho(\gamma))$ for some $\gamma \in \Gamma$, are invariant functions on $\Hom(\Gamma , \mathbf{G})$. 
 For most of classical groups $\mathbf{G}$, the character variety $\mathfrak{X}_\mathbf{G}(\Gamma)$ can be identified with a subvariety of some $\CC^m$ using trace functions. This justifies the terminology, as well as the name \emph{character map} for the projection $\chi \colon \Hom(\Gamma, \mathbf{G}) \to  \mathfrak{X}_\mathbf{G}(\Gamma)$, that we write $\rho \mapsto \chi_\rho$. 
 Indeed, whenever $\mathbf{G}$ is a classical complex group, such as $\mathrm{GL}_n(\CC)$, $\mathrm{SL}_n(\CC)$, $\mathrm{O}_n(\CC)$, $\mathrm{SO}_n(\CC)$ or $\mathrm{Sp}_{2n}(\CC)$, there are known families of generators for the ring of invariant functions 
 $ \CC[\Hom(\Gamma,\mathbf{G})]^{\mathbf{G}} \simeq
 \CC[ \mathfrak{X}_\mathbf{G}(\Gamma) ]$.
 %
 Procesi proves in \cite{procesi_invariant_1976} that the trace functions generate the ring of invariant functions for the groups $\mathrm{GL}_n(\CC)$ and $\mathrm{SL}_n(\CC)$ (Theorem 1.3), as well as $\mathrm{O}_n(\CC)$ (Theorem 7.1) and $\mathrm{Sp}_{2n}(\CC)$ (Theorem 10.1). 
 Moreover, he gives a (very large) generating set of the ring of invariant functions depending on a generating set for $\Gamma$. More efficient generating sets are given by Sikora in \cite{sikora_generating_2013}, where he also proves
 that the trace functions are generators for $\mathrm{SO}_{2n+1}(\CC)$-character varieties.

 For $\mathrm{SO}_{2n}(\CC)$-representations, the set of trace functions does no longer generate the ring of invariant functions, as pointed out by Sikora in \cite{sikora_so2nccharacter_2017}. However, following \cite{sikora_generating_2013}, 
 the ring $\CC[ \mathrm{Hom}(\Gamma,\mathrm{SO}_{2n}(\CC))]^{\mathrm{SO}_{2n}(\CC)}$ is generated by the trace functions and by the functions $Q_{\gamma_1,\dots,\gamma_n}$, which are defined as follows for $\gamma_1,\dots,\gamma_n \in \Gamma$.
 Consider the function $Q_{2n} : (\mathcal{M}_{2n}(\CC))^n \to \CC$, defined by
 \[
 Q_{2n}(A_1,\dots,A_m) = 
 \sum_{\sigma \in \mathfrak{S}_{2m}} 
 \epsilon(\sigma)
 \prod_{i=1}^{m}
 ((A_i)_{\sigma(2i-1)\sigma(2i)}
 -(A_i)_{\sigma(2i)\sigma(2i-1)}).
 \]
 Observe that the function $Q_m$ is invariant under $\mathrm{SO}_{2n}(\CC)$-conjugation, but not by  $\mathrm{O}_{2n}(\CC)$-conjugation.
 Hence, function $Q_{2n}$ induces an invariant function $Q_{\gamma_1,\dots,\gamma_n}$ on $\mathrm{Hom}(\Gamma,\mathrm{SO}_{2n}(\CC))$ defined by
 \[Q_{\gamma_1,\dots,\gamma_n}(\rho) = Q_{2n}(\rho(\gamma_1),\dots,\rho(\gamma_n)).\]
 
 We consider also the real points of $\mathbf{G}$, $\Hom(\Gamma, \mathbf{G})$ and $\mathfrak{X}_\mathbf{G}(\Gamma)$, given by a real form $G$ of $\mathbf{G}$ and its action on $\Hom(\Gamma,G)$. Observe that two twisted real forms give the same involutive antiholomorphic automorphism of $\mathfrak{X}_\mathbf{G}(\Gamma)$.
 
 The real forms of the classical groups are defined by involutions of the form $A \mapsto J \bar{A} J^{-1}$ or $A \mapsto J {}^t\! \bar{A}^{-1} J^{-1}$ (see for example the classification in \cite{onishchik_lie_1994}). Observe that they are all twisted for all the groups except $\mathrm{SL}_n(\CC)$ and $\mathrm{GL}_n(\CC)$, where the two classes are different.
 By \Cref{prop:complex_conjug_compatible_GIT}, these involutions induce two real structures on $\mathfrak{X}_\mathbf{G}(\Gamma)$, defined by involutions $\Phi_1$ and $\Phi_2$.
 On the one hand, using trace functions (and possibly functions of the form $Q_{\gamma_1,\dots,\gamma_n}$) as coordinates to embed $\mathfrak{X}_\mathbf{G}(\Gamma)$ into some $\CC^m$, the involution $\Phi_1$ is induced by the usual complex conjugation on $\CC^m$.
 On the other hand, the involution $\Phi_2$, for $\mathrm{GL}_n(\CC)$ and $\mathrm{SL}_n(\CC)$ is given as follows.
 Let $\gamma_1, \dots , \gamma_s \in \Gamma$ be such that the map
 \[
 \begin{aligned}
 \Hom(\Gamma,\mathbf{G}) &\to \CC^{2s} \\
 \rho &\mapsto \left(\tr(\rho(\gamma_1)) , \dots , \tr(\rho(\gamma_s)) , 
 \tr(\rho(\gamma_1^{-1})) , \dots , \tr(\rho(\gamma_s^{-1})) \right)
 \end{aligned}
 \]
 induces an embedding of $\mathfrak{X}_\mathbf{G}(\Gamma)$ into $\CC^{2s}$. Then $\Phi_2$ is induced by the map
 \[
 \begin{aligned}
 \CC^{2s} &\to \CC^{2s} \\
 (z_1, \dots , z_s , w_1 , \dots , w_s)
 & \mapsto ( \bar{w_1} , \dots , \bar{w_s} , \bar{z_1} , \dots \bar{z_s} )
 \end{aligned}.
 \]
 
 Whenever the real form $G$ of one of the complex groups is the compact real form, the topological quotient $\Hom(\Gamma, K)/K$ is also compact. When $K = \mathrm{SU}(n)$, Procesi and Schwartz prove in \cite{procesi_inequalities_1985} that $\Hom(\Gamma, G)/G$ is a semi-algebraic set.
 Furthermore, if $K$ is a maximal compact subgroup of a complex or real algebraic reductive
 group $\mathbf{G}$, some recent results prove there is a strong deformation retraction from the set $\Hom(\Gamma, \mathbf{G}) \mathbf{ \sslash G}$ to the quotient $\Hom(\Gamma, K)/K$. This fact is proven for Abelian groups by
 Florentino and Lawton in \cite{florentino_topology_2014}, for free groups by Casimiro, Florentino, Lawton and Oliveira
 in \cite{casimiro_topology_2014}, and for nilpotent groups by Bergeron in \cite{bergeron_topology_2015}.

 \subsection{Lifting good characters}
 
 Consider a classical complex group $\mathbf{G}$ and a real character $\chi_\rho \in \mathfrak{X}_\mathbf{G}(\Gamma)$ for one of the involutions $\Phi_1$ or $\Phi_2$ defined above. We are going to use \Cref{thm:lift_real_irred} to lift $\chi_\rho$ to a representation $\rho'$ with values in some real form $G'$ of $\mathbf{G}$, up to $G'$-conjugation.
 To do so, we first need a condition implying that the character is \emph{neat} in the sense of \Cref{def:irred_point_quotient}.
 
 The notion that seems adapted in this setting is the one of a \emph{good} representation. Following Sikora in \cite[Sections 3 and 4]{sikora_character_2012a} and Johnson and Millson in \cite{johnson_deformation_1987}, a representation $\rho \in \Hom(\Gamma , \mathbf{G})$ is \emph{good} if its $\mathbf{G}$-orbit is closed and the stabilizer of $\rho$ in $\mathbf{G}$ is the center of $\mathbf{G}$.
 These conditions are precisely conditions (\ref{cond_def:irred_closed_orbit}) and (\ref{cond_def:irred_central_stabilizer}) of \Cref{def:irreducible_point_X}, so a representation $\rho \in \Hom(\Gamma, \mathbf{G})$ is a neat point in this sense if and only if it is good and a regular point of $\Hom(\Gamma, \mathbf{G})$.
 
 Furthermore, if a representation is good then it is irreducible, but the converse is not always true. A convenient notion that implies goodness is the $\Ad$-irreducibility.
 Let $\mathfrak{g}$ be the Lie algebra of $\mathbf{G}$. Then, the adjoint representation of $\mathbf{G}$ induces a natural map $\Hom(\Gamma , \mathbf{G}) \to \Hom(\Gamma , \mathrm{GL}(\mathfrak{g}))$. We say that a representation $\rho \in \mathrm{Hom}(\Gamma , \mathbf{G})$ is \emph{$\Ad$-irreducible} if $\Ad \circ \rho \colon \Gamma \to \mathrm{GL}(\mathfrak{g})$ is irreducible. Thus, we have, for $\rho \in \Hom(\Gamma,\mathbf{G})$,
 \[
 \rho \text{ is } \Ad \text{-irreducible }
 \implies
 \rho \text{ is good }
 \implies
 \rho \text{ is irreducible}.
 \]
 The first implication is given by Schur's lemma, and the three statements are equivalent if $\mathbf{G}$ is $\mathrm{GL}_n(\CC)$ or $\mathrm{SL}_n(\CC)$. However, as pointed out by Sikora in \cite[Section 4]{sikora_character_2012a}, for the other classical complex groups there are irreducible representations that are not good.
 Observe also that, if $\mathbf{G} \subset \mathrm{GL}_n(\CC)$ is a linear irreducible subgroup, then $\rho \in \Hom(\Gamma , \mathbf{G})$ is $\Ad$-irreducible if and only if it is irreducible when seen as a representation with values in $\mathrm{GL}_n(\CC)$.
 Considering good representations, we prove the following theorem for lifting good real characters to representations with values in a real form.
 
  \begin{thm}\label{thm:lift_real_characters}
  	Let $\Gamma$ be a finitely generated group and
	$\mathbf{G}$ a connected complex reductive algebraic group.
	Let $\Phi$ be an antiholomorphic involution of $\mathfrak{X}_{\mathbf{G}}(\Gamma)$ induced by a real form $G$ of $\mathbf{G}$.
  	Let $\rho \in \Hom(\Gamma , \mathbf{G})$ be a good representation such that $\Phi(\chi_\rho) = \chi_\rho$.
  	Then, there exists a real form $G'$ of $\mathbf{G}$ such that 
  	$\rho \in \Hom(\Gamma,G')$.
  \end{thm}
  
  \begin{proof}
  	We can suppose, without lost of generality, that $\Gamma$ is a free group. Indeed, if we fix a set of generators for $\Gamma$ and $\Gamma'$ is the free group on these generators, we have the closed $\mathbf{G}$-equivariant immersion $\Hom(\Gamma , \mathbf{G}) \hookrightarrow \Hom(\Gamma', \mathbf{G})$.
  	Thus, we have also an inclusion $ \mathfrak{X}_\mathbf{G}(\Gamma) \subset \mathfrak{X}_\mathbf{G}(\Gamma')$.
  	Hence, if, given a real character $\chi_\rho \in \mathfrak{X}_\mathbf{G}(\Gamma)$, we can lift it to $\rho' \in \Hom(\Gamma', G')$ for some real form $G'$ of $\mathbf{G}$, then the lift will be in $\Hom(\Gamma, G')$.
  	
  	Now, note that if $\Gamma$ is a free group of order $s$, then $\Hom(\Gamma , \mathbf{G}) \simeq \mathbf{G}^s$ is irreducible, and all its points are regular. Since we suppose that $\rho$ is good, it is a neat point of $\Hom(\Gamma , \mathbf{G})$.
  	Thus, we can apply \Cref{thm:lift_real_irred}, and obtain real forms $G'$ of $\mathbf{G}$ and $Y'$ of $\Hom(\Gamma, \mathbf{G})$ such that $G'$ is a twisted form of the original real form $G$ of $\mathbf{G}$, $Y'$ is $G'$-invariant and $\rho \in Y'$.
  	Furthermore, there is $h \in \mathbf{G}$ such that $G'$ is the set of fixed points of the involution $g \mapsto h\bar{g}h^{-1}$ and $Y'$ is the set of fixed points of the involution $\rho' \mapsto h \cdot \bar{\rho}'$. Since $\mathbf{G}$ acts by conjugation, $Y'$ is precisely $\Hom(\Gamma, G')$.
  \end{proof}
  
  Since we can apply the previous theorem for classical complex groups and the fact of being good is implied by $\Ad$-irreducibility, we obtain immediately the following corollary.
  
   \begin{cor}
   	Let $\mathbf{G}$ be a classical complex group. Let $\Phi = \Phi_1$ or $\Phi_2$ be an involutive antiholomorphic automorphism of $\mathfrak{X}_{\mathbf{G}}(\Gamma)$.
   	Let $\rho \in \Hom(\Gamma , \mathbf{G})$ be an $\Ad$-irreducible representation such that $\Phi(\chi_\rho) = \chi_\rho$.
   	Then, there exists a real form $G'$ of $\mathbf{G}$ such that 
   	$\rho \in \Hom(\Gamma,G')$.
   	Moreover, $G'$ is a twisted form of the split real form of $\mathbf{G}$ if $\Phi = \Phi_1$ and a twisted form of the compact real form of $\mathbf{G}$ if $\Phi = \Phi_2$.
   \end{cor}

\section{A detailed example}\label{section:example}
 We study here the particular example of the $\mathrm{SL}_3(\CC)$-character variety of $\ZZ$, that we write $\mathfrak{X}_{\mathrm{SL}(3,\CC)}(\ZZ)$.
 Since $\Hom(\ZZ , \mathrm{SL}_3(\CC) ) \simeq \mathrm{SL}_3(\CC)$, this variety is the GIT quotient of $\mathrm{SL}_3(\CC)$ by its inner automorphisms.
 We identify first the complex character variety, isomorphic to $\CC^2$, and describe its real points for the two involutive antiholomorphic automorphisms induced by the real forms of $\mathrm{SL}_3(\CC)$. Then, we study the link between those real points and the real GIT quotients for the real forms $\mathrm{SL}_3(\RR)$, $\mathrm{SU}(3)$ and $\mathrm{SU}(2,1)$.
 
 In order to avoid heavy notation, we keep the general terms for GIT quotients. Therefore, we let $\mathbf{X} = \Hom(\ZZ, \mathrm{SL}_3(\CC) ) \simeq \mathrm{SL}_3(\CC)$ and $\mathbf{G} = \mathrm{SL}_3(\CC)$. The group $G$ will be one of the real forms $\mathrm{SL}_3(\RR)$, $\mathrm{SU}(3)$ or $\mathrm{SU}(2,1)$, and $X = \Hom(\ZZ,G) \simeq G$.

\subsection{The complex character variety}
 We are going to identify the character variety $\mathfrak{X}_{\mathrm{SL}(3,\CC)}(\ZZ)$ with $\CC^2$, as the image of the map $M \mapsto (\tr(M), \tr(M^{-1}))$.
 First, observe that if $M \in \mathrm{SL}_3(\CC)$, its characteristic polynomial is $X^3 - \tr(M)X^2 + \tr(M^{-1})X -1$, so the function $\tr(M^k)$ is a polynomial in $\tr(M)$ and $\tr(M^{-1})$ for all $k \in \ZZ$ by the Newton identities.
 Since the $\mathrm{SL}_3(\CC)$-invariant functions of $\mathrm{Hom}(\Gamma, \mathrm{SL}_3(\CC))$ are generated by the trace functions, they are generated by $\tr(M)$ and $\tr(M^{-1})$.
 Moreover, the map $M \mapsto (\tr(M), \tr(M^{-1}))$ is surjective,
 so $\mathfrak{X}_{\mathrm{SL}(3,\CC)}(\ZZ)$ is isomorphic to $\CC^2$.

 Note also that the orbit of $M \in \mathrm{SL}_3(\CC)$ by conjugation is closed if and only if $M$ is diagonalizable.
 Thus, the closed orbits are classified by the spectrum counted with multiplicity; a lift in $\mathrm{SL}_3(\CC)$ of the point $(z,w) \in \CC^2$ is a diagonal matrix with entries the roots of $X^3 -zX^2 + wX -1$.
 
 There are three types of closed orbits, depending on the multiplicities of the eigenvalues. The corresponding stabilizers, that we will use to determine the number of lifts of real points, are in \Cref{table:stabilizers}. Note that there are no neat points, since no stabilizer is trivial.
 
  \begin{table}
  	\centering
  	\[
  	\begin{array}{c c}
  	\toprule
  		\text{Semisimple element}
  		& \text{Stabilizer in } \mathrm{SL}_3(\CC) \\ 
  		\midrule 
  		\begin{psmallmatrix}
  		a & 0 & 0 \\
  		0 & b & 0 \\
  		0 & 0 & (ab)^{-1} 
  		\end{psmallmatrix}
  		&
  		\left\{
  		\begin{psmallmatrix}
  		u & 0 & 0 \\
  		0 & v & 0 \\
  		0 & 0 & (uv)^{-1}
  		\end{psmallmatrix}
  		\mid
  		u,v \in \CC^*
  		\right\}
  		\simeq (\CC^*)^2 \\ 
  		\midrule[0pt]
  		\begin{psmallmatrix}
  		a & 0 & 0 \\
  		0 & a & 0 \\
  		0 & 0 & a^{-2} 
  		\end{psmallmatrix}
  		&
  		\left\{
  		\begin{psmallmatrix}
  		M & 0  \\
  		0 & \det(M)^{-1}
  		\end{psmallmatrix}
  		\mid
  		M \in \mathrm{GL}_2(\CC)
  		\right\}
  		\simeq \mathrm{GL}_2(\CC) \\ 
  		\midrule[0pt]
  		\begin{psmallmatrix}
  		a & 0 & 0 \\
  		0 & a & 0 \\
  		0 & 0 & a 
  		\end{psmallmatrix}
  		&
  		\mathrm{SL}_3(\CC)\\ 
  		\bottomrule
  	\end{array}
  	\]
  	\caption{Stabilizers of diagonal matrices in $\mathrm{SL}_3(\CC)$.}\label{table:stabilizers}
  \end{table}

\subsection{Real forms of SL(3,C)}\label{subsect:real_points_sl3c_char_var}
Up to inner automorphisms, there are three real forms of $\mathrm{SL}_3(\CC)$, namely $\mathrm{SL}_3(\RR)$, $\mathrm{SU}(3)$ and $\mathrm{SU}(2,1)$. 
The corresponding complex conjugations are, respectively, the usual complex conjugation $M \mapsto \bar{M}$ and the involutions $\tau_1$ and $\tau_2$ given by
\begin{align*}
	\tau_1(M) &={}^t\! \bar{M}^{-1} &
	&\text{ and }&
	\tau_2(M) &=I_{2,1} {}^t\! \bar{M}^{-1} I_{2,1}
\end{align*}
where $I_{2,1} =
\begin{psmallmatrix}
1 & 0 & 0 \\
0 & 1 & 0 \\
0 & 0 & -1
\end{psmallmatrix}
$.
The usual complex conjugation induces, in $\mathfrak{X}_{SL(3,\CC)}(\ZZ) \simeq \CC^2$, the involution $\Phi_1 : (z,w) \mapsto (\bar{z},\bar{w})$, while $\tau_1$ and $\tau_2$ both induce the involution $\Phi_2 : (z,w) \mapsto (\bar{w}, \bar{z})$.
Hence, for the fixed points of these involutions, we have:
\begin{align*}
	\mathfrak{X}_{\mathrm{SL}(3,\CC)}(\ZZ)^{\Phi_1} &\simeq \{(z,w)\in \CC^2 \mid z,w \in \RR \} \simeq \RR^2
	\\
	\mathfrak{X}_{\mathrm{SL}(3,\CC)}(\ZZ)^{\Phi_2} &\simeq \{(z,w)\in \CC^2 \mid w = \bar{z} \} \simeq \CC
\end{align*}

In order to have some information on the fibers of the $\beta$ maps, we compute the first cohomology sets of the stabilizers of diagonalizable matrices.

	There are two types of subgroups conjugated to $(\CC^*)^2$ in $\mathrm{SL}_3(\CC)$ for each of the real forms $\mathrm{SL}_3(\RR)$ and $\mathrm{SU}(2,1)$.
	In \Cref{table:cardinality_H1}, the groups $(\CC^*)^2$ of type $(a)$ are the stabilizers of elements diagonalizable in the group (with real eigenvalues in $\mathrm{SL}_3(\RR)$ and elliptic in $\mathrm{SU}(2,1)$)
	whereas the groups $(\CC^*)^2$ of type $(a)$ are the stabilizers of elements which are not diagonalizable in the group (with a non-real eigenvalue in $\mathrm{SL}_3(\RR)$ and loxodromic in $\mathrm{SU}(2,1)$).

\begin{lemme}\label{lemma:computation_h1}
	Let $\mathbf{H} \subset \mathrm{SL}_3(\CC)$ be one of the groups of \Cref{table:stabilizers}, and $\tau : \mathrm{SL}_3(\CC) \to \mathrm{SL}_3(\CC)$ be either the usual complex conjugation or one of the involutions $\tau_1$ or $\tau_2$.
	Then, the cardinal of $H^1_\tau(\galCR ,\mathbf{H} )$ is the value given in \Cref{table:cardinality_H1}.
\end{lemme} 

\begin{proof}
	If $\tau$ is the usual complex conjugation, and $\mathbf{H}$ is one of the groups $\mathrm{SL}_3(\CC)$, $\mathrm{GL}_2(\CC)$ or $(\CC^*)^2$ of type $(a)$ of \Cref{table:stabilizers}, then the action of $\tau$ on $\mathbf{H}$ is the usual complex conjugation. Hence, by \cite[X.1, corollay of Prop. 3]{serre_cohomologie_1994d}, $H^1_\tau(\galCR ,\mathbf{H})$ is trivial.
	If $\mathbf{H}$ is the group $(\CC^*)^2$ of type $(b)$,
	then $\tau$ acts by $\tau(u,v) = (\bar{v},\bar{u})$, so the set of cocycles is the set of points of the form $(\alpha , \bar{\alpha}^{-1})$. Since $(\alpha , \bar{\alpha}^{-1}) = (\alpha^{-1} , 1)^{-1} \tau(\alpha^{-1} , 1)$, we have that $H^1_\tau(\galCR ,\mathbf{H} )$ is trivial.
	
	We deal now with the involutions $\tau_1$ and $\tau_2$. Let $J \in \mathrm{GL}_3(\CC)$ be a diagonal matrix with entries equal to $\pm 1$.
	Let $\tau : M \mapsto J {}^t\! \bar{M}^{-1} J$ be an antiholomorphic involution, so $\tau = \tau_1$ if $J = \mathrm{Id}$ and $\tau = \tau_2$ if $J = I_{2,1}$. We have four cases, depending on $\mathbf{H}$.
	
	\emph{First case: $\mathbf{H} = (\CC^*)^2$ of type $(a)$:}
	A straightfoward computation gives that $\tau(u,v) = (\bar{u}^{-1} , \bar{v}^{-1})$, so the set of cocycles is $(\RR^*)^2$. Moreover, $(r,s) \sim (r',s')$ if and only if $rr' > 0$ and $ss' >0$. Thus, $H^1_\tau(\galCR ,\mathbf{H} ) \simeq \ZZ/2\ZZ \times \ZZ/2\ZZ$. Observe that, since $(\CC^*)^2$ is Abelian, we have a cohomology \emph{group}.
	
	\emph{Second case: $\mathbf{H} = (\CC^*)^2$ of type $(b)$:}
	A direct computation gives $\tau(u,v) = (\bar{u} \bar{v} , \bar{v}^{-1})$, so the set of cocycles is the set of points of the form $(\alpha , \abs{\alpha}^{-2})$. Since $(\alpha , \abs{\alpha}^{-2}) = (1 , \bar{\alpha} )^{-1} \tau(1 , \bar{\alpha} )$, we have that $H^1_\tau(\galCR ,\mathbf{H} )$ is trivial.
	
	\emph{Third case: $\mathbf{H} = \mathrm{SL}_3(\CC)$:}
	In this case, a cocycle is given by matrix $M \in \mathrm{SL}_3(\CC)$ such that $M \tau(M) = \mathrm{Id}$, i.e.\ such that
	$MJ$ is Hermitian.
	Observe that since $\det(M) = 1$, the signature of $MJ$ can only be $(0,3)$ or $(2,1)$.
	Two such matrices $M$ and $M'$ are equivalent if and only if there is $N \in \mathrm{SL}_3(\CC)$ such that
	$N(MJ) {}^t\!\bar{N} =  M'J$, which is equivalent to say that the non-degenerate Hermitian forms $MJ$ and $M'J$ are equivalent. Since Hermitian forms are classified by their signature, there are exactly $2$ equivalence classes in $H^1_\tau(\galCR ,\mathbf{H} )$.
	
	\emph{Fourth case: $\mathbf{H} = \mathrm{GL}_2(\CC)$:} Let $J'$ be the upper left $2 \times 2$ block of $J$. The restriction to this blocks gives an action of $\tau$ on $\mathrm{GL}_2(\CC)$ by $M \mapsto J'{}^t\! \bar{M}^{-1} J'$. By an analogous computation as the one of the third case, there is a one-to-one correspondence between $H^1_\tau(\galCR ,\mathbf{H} )$ and the signatures of non-degenerate $2 \times 2$ Hermitian matrices. Thus, $\abs{H^1_\tau(\galCR ,\mathbf{H} )} = 3$.
	
\end{proof}

\begin{table}
	\[
	\begin{array}{ccc}
	\toprule
		\text{Involution } \tau & \text{Subgroup } \mathbf{H} \subset \mathrm{SL}_3(\CC) & \abs{H^1_\tau(\galCR ,\mathbf{H} )} \\
	\midrule
		 & (\CC^*)^2 \text{ of type } (a) & 1 \\
	\multirow{2}{*}{$M \mapsto \bar{M}$}	 & (\CC^*)^2 \text{ of type } (b) & 1 \\
		 & \mathrm{GL}_2(\CC) & 1 \\
		 & \mathrm{SL}_3(\CC) & 1 \\
	\midrule
		 & (\CC^*)^2 \text{ of type } (a) & 4 \\
		\tau_1 \text{ or } \tau_2 & \mathrm{GL}_2(\CC) & 3 \\
		 & \mathrm{SL}_3(\CC) & 2 \\
	\midrule 
		 \tau_2 & (\CC^*)^2 \text{ of type } (b) & 1 \\
	\bottomrule
	\end{array}
	\]
	\caption{Cardinality of the first cohomology sets of stabilizers of diagonal matrices in $\mathrm{SL}_3(\CC)$.}\label{table:cardinality_H1}
\end{table}

\subsection{Real points for SL(3,R)}
 We deal first with the real form $\mathrm{SL}_3(\RR)$, and the usual complex conjugation. In our general notation, we have
 $X = \Hom(\ZZ, \mathrm{SL}_3(\RR)) \simeq \mathrm{SL}_3(\RR)$ and $G = \mathrm{SL}_3(\RR)$, 
 while $\mathbf{X} = \Hom(\ZZ, \mathrm{SL}_3(\CC)) \simeq \mathrm{SL}_3(\CC)$ and $\mathbf{G} = \mathrm{SL}_3(\CC)$.
 We identify $\mathbf{X \sslash G}$ with $\CC^2$, using the traces of a matrix and its inverse.
 Hence, the map $\pi \vert_X$ is given by 
 \[
 \pi\vert_X: 
 \begin{aligned}
 \Hom(\ZZ, \mathrm{SL}_3(\RR)) &\to (\mathbf{X \sslash G})(\RR)\simeq \RR^2 \\
 M &\mapsto (\tr(M), \tr(M^{-1}))
 \end{aligned}
  \]
 Recall that the projection $p: \Hom(\ZZ, \mathrm{SL}_3(\RR)) \to \Hom(\ZZ, \mathrm{SL}_3(\RR)) \sslash \mathrm{SL}_3(\RR)$ is surjective, and that we have a natural map $\beta : X \sslash G \to (\mathbf{X \sslash G})(\RR)$ satisfying \[\pi \vert_X = \beta \circ p.\]
 \begin{prop}
 	The map $\beta : \Hom(\ZZ, \mathrm{SL}_3(\RR)) \sslash \mathrm{SL}_3(\RR) \to \RR^2$ is a homeomorphism.
 \end{prop}
 \begin{proof}
 	Since $\beta$ is continuous and proper, and $\RR^2$ is locally compact and Hausdorff, we only need to prove that $\beta$ is bijective. 
 	
 	First, we prove the surjectivity of $\beta$.
 	Let $(r,s) \in \RR^2$, and consider the companion matrix $M =
 	\begin{psmallmatrix}
 	0 & 0 & 1 \\
 	1 & 0 & -s \\
 	0 & 1 & r
 	\end{psmallmatrix}
 	\in \mathrm{SL}_3(\RR)$.
 	The characteristic polynomial of $M$ is $X^3 - rX^2 + sX -1$, so $\tr(M) = r$ and $\tr(M^{-1}) = s$. Hence, $\pi(M) = (r,s)$  and 
 	$\pi \vert_X$ is surjective. Since $\pi \vert_X = \beta \circ p$, we obtain the surjectivity of $\beta$.
 	
 	Now, we prove that $\beta$ is injective. Let $(r,s) \in \RR^2$. 
 	We constructed, in \Cref{subsect:number_of_lifts}, an injective map $\varphi$ from $\beta^{-1}((r,s))$ to $H^1(\galCR, \mathbf{G}_M)$, where $\mathbf{G}_M$ is the stabilizer in $\mathrm{SL}_3(\CC)$ of a diagonalizable matrix $M \in \pi^{-1}((r,s))$.
	The possible groups $\mathbf{G}_M$ are listed in \Cref{table:stabilizers}, and are isomorphic to $(\CC^*)^2$, $\mathrm{GL}_2(\CC)$ or $\mathrm{SL}_3(\CC)$. Note that the induced complex conjugation on these subgroups is the usual complex conjugation.
	The corresponding $H^1(\galCR , \mathbf{H})$, computed in \Cref{table:cardinality_H1} and \Cref{lemma:computation_h1}, are trivial.
	Hence, in all the cases, the injective map $\varphi$ has a singleton as target. Thus, all the fibers of $\beta$ are singletons, i.e.\ $\beta$ is injective.
 \end{proof}

\subsection{Real points for SU(3)}
 We consider here the real form $G = \mathrm{SU}(3)$, acting on $X = \mathrm{SU}(3)$ by conjugation. The corresponding complex conjugation is $\tau_1$.
 By \Cref{prop:compact_real_form}, we know that the map $\beta : X\sslash G \to \CC$ is a homeomorphism onto its image.
 Observe that the eigenvalues of a matrix in $\mathrm{SU}(3)$ have modulus $1$, and that given $a,b,c \in \CC^*$ of modulus $1$ and such that $abc = 1$, the diagonal matrix with entries $(a,b,c)$ is unitary.
 Therefore, the image of $\beta$ is precisely the set of $z \in \CC$ such that 
 $X^3 -zX^2 + \bar{z}X - 1$ has all its roots of modulus $1$. By considering the resultant of this polynomial and its derivative we obtain the following function defined by Goldman in \cite[Section 6.2.3]{goldman}
 \begin{equation}\label{eq:goldman_function}
 	f(z) = \abs{z}^4 - 8\Re(z^3) + 18\abs{z}^2 - 27. \tag{$*$}
 \end{equation}
 Hence, the image of $\beta$ is the set $f^{-1}(\RR^-)$, which is the inner part of the curve of \Cref{fig:curve_goldman}.
 Therefore, the $\mathrm{SU}(3)$-character variety for $\ZZ$ is homeomorphic to a closed triangle.
 
 \begin{figure}[htbp]
 	\includegraphics[width=4cm]{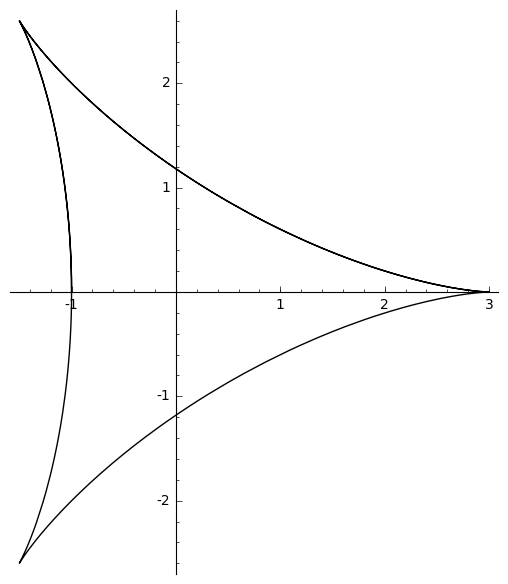}
 	\caption{The zero level set of $f$.} \label{fig:curve_goldman}
 \end{figure}
 
 We make a last observation about the map $\varphi$ defined in \Cref{subsect:number_of_lifts}, that is an injective map from the fibers of $\beta$ to the first cohomology set of the corresponding stabilizer. Since here $\beta$ is bijective, the image of $\varphi$ is always a singleton. However, the possible stabilizers are given in \Cref{table:stabilizers}, and the corresponding cohomology sets, given in \Cref{table:cardinality_H1} have cardinal $2$, $3$ or $4$. Thus, the map $\varphi$ is never surjective in this case.

 \subsection{Real points for SU(2,1)}
  We consider now the action by conjugation of $G = \mathrm{SU}(2,1)$ on the space $X = \Hom(\ZZ, \mathrm{SU}(2,1)) \simeq \mathrm{SU}(2,1)$.
  The group $\mathrm{SU}(2,1)$ acts by isometries on the complex hyperbolic space $\HH^2_\CC$, so its elements can be classified by their dynamics. In order to understand the different quotients of $\Hom(\ZZ, \mathrm{SU}(2,1))$ by conjugation, we begin by describing the elements of $\mathrm{SU}(2,1)$.
 
  \subsubsection*{Elements of SU(2,1) up to conjugation}
  The elements of $\mathrm{SU}(2,1)$ can be classified, up to conjugation in $\mathrm{SU}(2,1)$, by their dynamics on the complex hyperbolic plane. See for example Chapter 6 of Goldman's book \cite{goldman} or the thesis of Genzmer \cite{genzmer}. We will give a representative of each conjugacy class.
  For clearness on the representatives, we will use two different Hermitian forms $\Psi$. They are given by the matrices
  \begin{align*}
  H_1 &= 
  \begin{psmallmatrix}
  1 & 0 & 0  \\
  0 & 1 & 0  \\
  0 & 0 & -1
  \end{psmallmatrix}
  &
  H_2
  &=
  \begin{psmallmatrix}
  0 & 0 & 1 \\
  0 & 1 & 0 \\
  1 & 0 & 0
  \end{psmallmatrix}
  \end{align*}
  
  \paragraph{Elliptic elements}
  If an element $U \in \mathrm{SU}(2,1)$ has an eigenvector $v$ such that $\Psi(v) < 0$, $U$ is called \emph{elliptic}. When considering the Hermitian form $H_1$, it is conjugated in $\mathrm{SU}(2,1)$ to a matrix of the form
  \[
  E_{(a,b,c)}
  =
  \begin{psmallmatrix}
  e^{ia} & 0          & 0           \\
  0           & e^{ib} & 0           \\
  0           & 0          & e^{ic}
  \end{psmallmatrix}
  \]
  where $a,b,c \in \RR$ and $a+b+c = 0$. If $a$, $b$ and $c$ are considered modulo $2\pi$,
  two such matrices $E_{(a,b,c)}$ and $E_{(\alpha',\beta',\gamma')}$ are conjugated in $\mathrm{SL}_{3}(\CC)$ if and only if $(a',b',c')$ is a permutation of $(a,b,c)$.
  However, they are conjugated $\mathrm{SU}(2,1)$ if and only if $c = c'$ and $(a',b') = (a,b)$ or $(b,a)$. The nontrivial implication in this last claim is given by the following lemma.
  
  \begin{lemme}\label{lemma:negative_eigenvalue_continuous}
  	Let $\tilde{\mathcal{E}} \subset \mathrm{SU}(2,1)$ be the set of elliptic elements. For $U \in \tilde{\mathcal{E}}$, let $\mathrm{eig}^-(U)$ be the eigenvalue of $U$ for an eigenvector $v$ such that $\Psi(v)<0$. Then $\mathrm{eig}^- : \tilde{\mathcal{E}} \to S^1$ is a well defined continuous function that is invariant by $\mathrm{SU}(2,1)$-conjugation.
  \end{lemme}
  \begin{proof}
  	First, observe that if $U \in \tilde{\mathcal{E}}$ and $v_1 , v_2 \in \CC^3$ are eigenvectors such that $\Psi(v_1) \leq 0$ and $\Psi(v_2) \leq 0$, then they have the same eigenvalue, which has modulus $1$.
  	%
  	Thus, the function $\mathrm{eig}^- : \tilde{\mathcal{E}} \to S^1$ is well defined. Since $\mathrm{SU}(2,1)$-conjugation preserves the eigenvalues and the sign of $\Psi$ on eigenvectors, the function is invariant by conjugation. It only remains to prove that it is continuous.
  		
  	We reason by contradiction. Suppose that $(U_n)_{n\in \NN}$ is a sequence of $\tilde{\mathcal{E}}$ converging to $U \in \tilde{\mathcal{E}}$ such that $\mathrm{eig}^-(U_n) \not\to \mathrm{eig}^-(U)$. Consider eigenvectors $v_n$ for $U_n$ in the unit sphere of $\CC^3$ such that $\Psi(v_n)<0$. 
  	Since $S^1$ and the unit sphere of $\CC^3$ are compact, maybe after taking a subsequence we can suppose that $v_n \to v \in \CC^3$ and that $\mathrm{eig}^-(U_n) \to \lambda \neq \mathrm{eig}^-(U)$. Furthermore, we have $\Psi(v) \leq 0$. But since $U_nv_n = \mathrm{eig}^-(U_n) v_n$, we obtain that $Uv = \lambda v$. Since $U \in \tilde{\mathcal{E}}$, we obtain that $\lambda = \mathrm{eig}^-(U)$, which is a contradiction.
  \end{proof}
  		
 \paragraph*{Loxodromic and parabolic elements}
  If one of the eigenvalues of $U$ has modulus $\neq 1$, then, when considering the Hermitian form $H_2$, $U$ is called \emph{loxodromic} and is conjugated in $\mathrm{SU}(2,1)$ to a matrix of the form
	\[
	L_\lambda =
	\begin{psmallmatrix}
	\lambda & 0                        & 0                     \\
	0       & \conjug{\lambda}/\lambda & 0                     \\
	0       & 0                        & \conjug{\lambda}^{-1}
	\end{psmallmatrix}
	\]
	where $\lambda \in \CC$ has modulus $<1$. If $\abs{\lambda} = 1$, we still have a unitary element for the Hermitian form $H_2$ with a double eigenvalue, that is elliptic.
	If $U$ is not elliptic nor loxodromic, it is called \emph{parabolic}, and is conjugated in $\mathrm{SU}(2,1)$ to
	\[
	P_{\alpha,z,t} =
	e^{i\alpha}
	\begin{psmallmatrix}
	1 & -e^{-3i\alpha}\conjug{z} & -\frac{1}{2}(\abs{z}^2 + it) \\
	0 & e^{-3i\alpha}            & z                            \\
	0 & 0                        & 1
	\end{psmallmatrix}
	\]
	where $\alpha,t \in \RR$ and $z \in \CC$.
  		
  \subsubsection*{Real points of the complex GIT quotient}
   Recall, from \Cref{subsect:real_points_sl3c_char_var}, that the character variety $\mathfrak{X}_{\mathrm{SL}(3,\CC)}(\ZZ)$ is isomorphic to $\CC^2$. The fixed points for the involution induced by $\mathrm{SU}(2,1)$ are the points of the form $(z,\bar{z})$, that we identify with $\CC$ using the first projection. Then, the corresponding projection map $\pi : \mathrm{SU}(2,1) \to \CC$ is precisely the trace. By considering the representatives of elliptic and loxodromic elements, we obtain that $\tr : \mathrm{SU}(2,1) \to \CC$ is surjective.
   Thus, $\mathfrak{X}_{\mathrm{SU}(2,1)}(\ZZ) \simeq \CC$, and the map
   $\beta: X \sslash G \to \CC$ is surjective.
   
  \subsubsection*{The real GIT quotient}
	We study now the real GIT quotient $X \sslash G$. 
	Observe that the closed orbits for the $\mathrm{SU}(2,1)$-action, which correspond to diagonalizable elements, are precisely the ones of elliptic and loxodromic elements.
	Let $\mathcal{L}$ and $\mathcal{E}$ be the subsets of $X \sslash G$ obtained by projecting the loxodromic and elliptic elements of $\mathrm{SU}(2,1)$ respectively. Thus, we have a disjoint union $X \sslash G = \mathcal{L} \cup \mathcal{E}$.
	We study now the topology of $\mathcal{L}$ and $\mathcal{E}$.

  	\begin{prop}
  		We have $\mathcal{L} \simeq \{\lambda \in \CC^* \mid \abs{\lambda} < 1 \}$.
  	\end{prop}	
  \begin{proof}
  	By the description given above, the set $\{ L_\lambda \mid \lambda \in \CC^*, \,  |\lambda|<1 \}$ contains exactly one representative of each $\mathrm{SU}(2,1)$ conjugacy class.
  \end{proof}				 
				
  \begin{prop}
  	The space $\mathcal{E}$ is homeomorphic to a projective plane minus an open disk.
  \end{prop}
  \begin{proof}
  	Since all elliptic elements are diagonalizable, we know that $\mathcal{E}$ is precisely the space of $\mathrm{SU}(2,1)$-orbits of elliptic elements. Recall that we gave an explicit representative of the form $E_{(a,b,c)}$ of each conjugacy class.
  	We parametrize the elliptic elements $E_{(a,b,c)}$ by $(a,b) \in [0,2 \pi ]^2$, so there is at least an element in each conjugacy class. In order to obtain $\mathcal{E}$, we need to do the following identifications:
  	\begin{align*}
  	(a,b) \sim (b,a) &  & (0,b) \sim (2\pi,b) &  & (a,0) \sim (a, 2\pi)
  	\end{align*}
  	We obtain a real projective plane minus a disk, as in \Cref{fig:espace_he}.
  	\end{proof}
  	\begin{figure}
  		\centering
  		\includegraphics[width = 4cm]{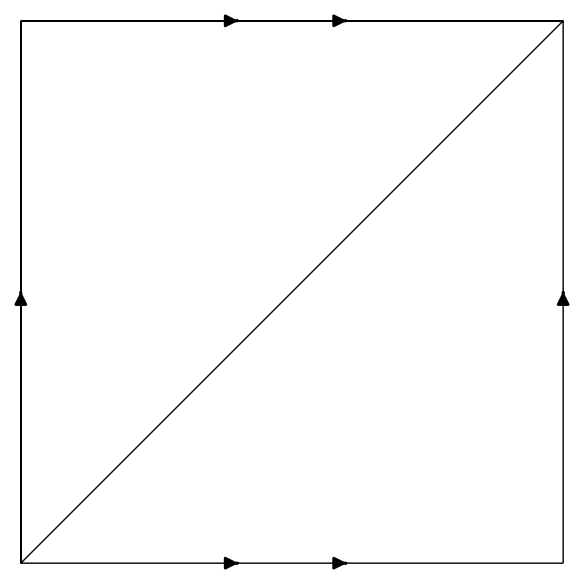}
  		\hspace{2cm}
  		\includegraphics[width = 4cm]{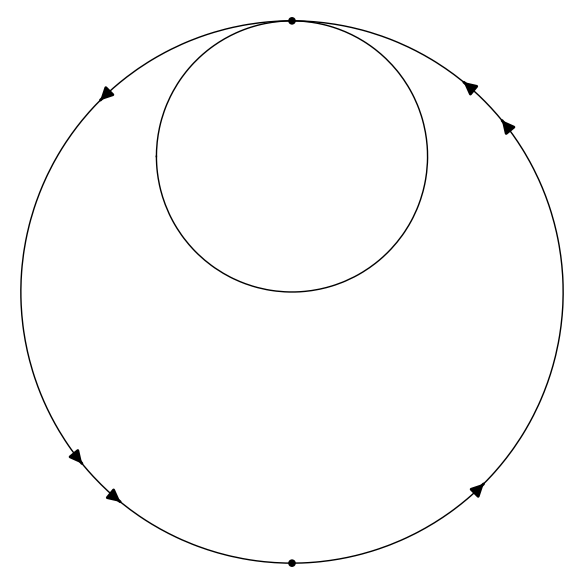}
  		\caption{The space $\mathcal{E}$.}\label{fig:espace_he}
  	\end{figure}

	In oder to complete the description of $X \sslash G = \mathcal{L} \cup \mathcal{E}$, it only remains to understand the gluing of $\mathcal{L}$ and $\mathcal{E}$ along their boundaries.
	The boundaries of $\mathcal{L}$ and $\mathcal{E}$ correspond to the matrices having a double eigenvalue. In the parametrization of $\mathcal{L}$ by $\{\lambda \in \CC^* \mid \abs{\lambda} < 1 \}$, the boundary is $\{\lambda \in \CC \mid \abs{\lambda} = 1 \}$, and corresponds to elliptic and ellipto-parabolic elements.
	An elliptic element with a double eigenvalue is conjugated to either $E_{(a,a,-2a)}$ (called reflection on a point) or $E_{(a,-2a,a)}$ (called reflection on a line). On the one hand, a reflection on a point cannot be a limit of loxodromic elements, since it has an isolated fixed point $[v] \in \CC\mathbb{P}^2$ such that $\Psi(v)<0$.
	On the other hand, when $\lambda \to e^{ia}$, $L_\lambda$ converges to
	an elliptic element with a double eigenvalue $e^{i a}$, conjugated to $E_{(a,-2a,a)}$.
	Thus, the identification of the boundaries of $\mathcal{L}$ and $\mathcal{E}$ is given by identifying the point of parameter $e^{ia}$ with the orbit of $E_{(a,-2a,a)}$.
	We deduce a full description of the space $X \sslash G$. A picture of the gluing is given in \Cref{fig:gluing_instructions}.
	
	\begin{prop}
		The space $X \sslash G$ is homeomorphic to the gluing of a real projective plane $\RR \mathbb{P}^2$ minus an open disk  together with a punctured disk. The gluing is along a circle in the nontrivial class of $\pi_1(\RR \mathbb{P}^2)$.
		In particular, $X \sslash G$ is not a manifold and has a boundary homeomorphic to $S^1$.
	\end{prop}

	\begin{proof}
	 The only remaining point is to identify the points of $\mathcal{E}$ coming from the orbits of $E_{(a,-2a,a)}$. It is a circle which is not homotopically trivial, drawn in \Cref{fig:espace_he_regions}, that cuts $\mathcal{E}$ into three triangles.
	\end{proof}
	
	\begin{figure}
		\centering
		\includegraphics[width = 4cm]{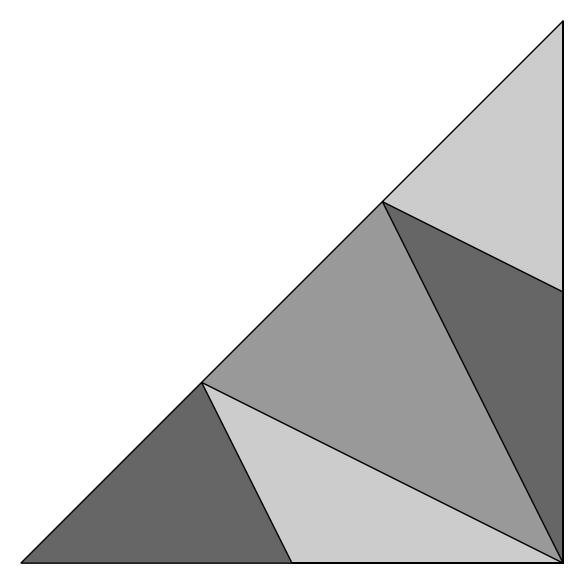}
		\hspace{2cm}
		\includegraphics[width = 4cm]{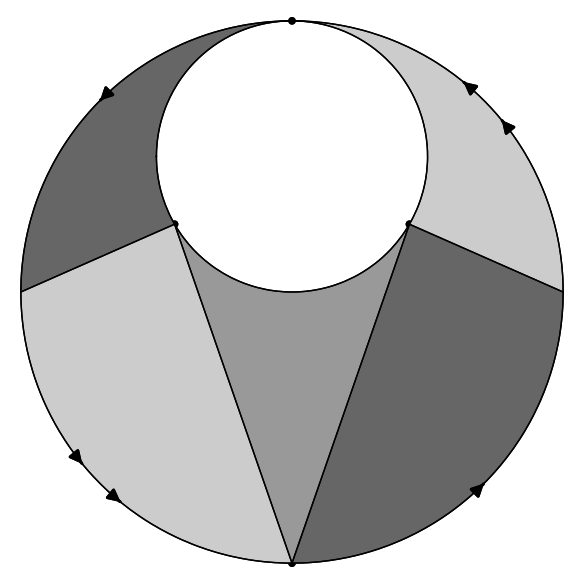}
		\caption{The space $\mathcal{E}$ cut into regions.}\label{fig:espace_he_regions}
	\end{figure}
		
	\begin{figure}[htbp]
		\centering
		\includegraphics[width = 4cm]{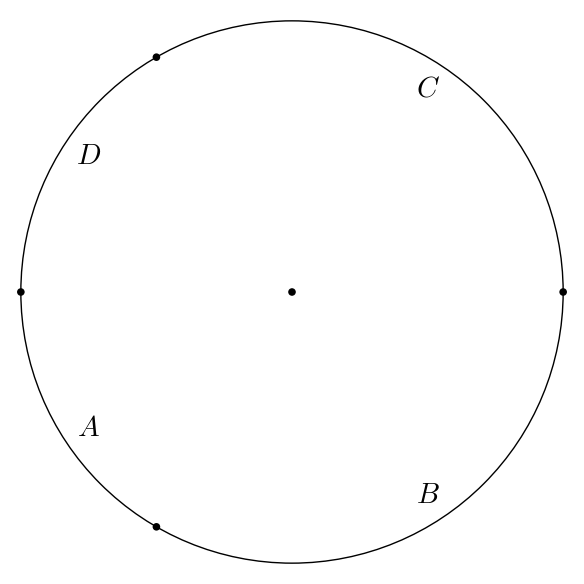}
		\hspace{2cm}
		\includegraphics[width = 4cm]{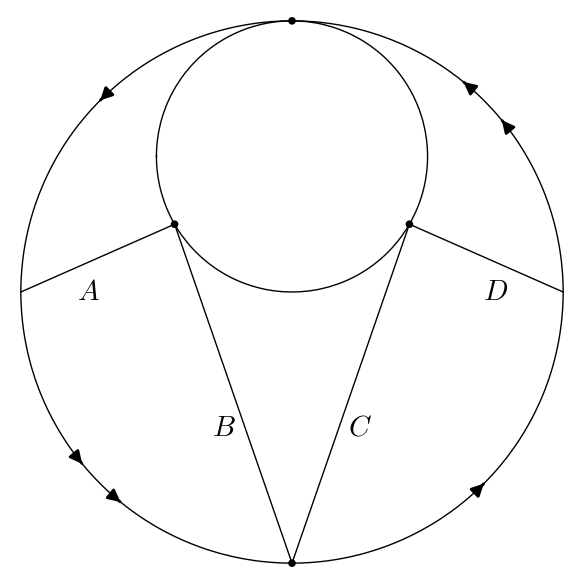}
		\caption{The gluing between $\mathcal{L}$ and $\mathcal{E}$.}\label{fig:gluing_instructions}
	\end{figure}
		
\subsubsection*{The quotient map}
 We consider now the map $\beta : X \sslash G \to \CC$ and the cardinal of its fibers; the behavior of the map and its fibers is quite curious.
 The function $f$ defined in \eqref{eq:goldman_function} is useful to identify the type of an element by its trace.
 \begin{prop}[Theorem 6.2.4 of \cite{goldman}]
 	Let $U \in \mathrm{SU}(2,1)$. Then
 	\begin{itemize}
 		\item The element $U$ is loxodromic if and only if $f(\tr(U)) > 0$.
 		\item If $U$ is parabolic, then $f(\tr(U)) = 0$.
 		\item If $f(\tr(U))<0$ then $f$ is elliptic.
 	\end{itemize}
 \end{prop}
 Let $\Delta = f^{-1}(\RR^*_-) \subset \CC$ be the open bounded component of \Cref{fig:curve_goldman}. By the previous proposition, we know that
 $\beta(\mathcal{L}) = \CC \setminus \bar{\Delta}$ is the unbounded open component of \Cref{fig:curve_goldman}, while $\beta(\mathcal{E}) = \bar{\Delta}$ is the closed bounded component. Observe that the lifts of $\partial \Delta$ are precisely the orbits of elliptic elements having a double eigenvalue, and that the points $3$, $3 \omega$ and $3 \omega^2$, where $\omega$ is a cube root of $1$, are the projections of the center of $\mathrm{SU}(2,1)$. Thus, we know precisely the number of lifts of each point of $\CC$. The quantities are summarized in \Cref{table:nb_orbit_lift_SU21}. 
 
 In \Cref{table:nb_orbit_lift_SU21}, for each point $x \in \CC$ we also record the cardinalities of the cohomology sets $H^1_\tau(\galCR, \mathbf{G}_{\tilde{x}} )$, where $\mathbf{G}_{\tilde{x}}$ is the stabilizer in $\mathrm{SL}_3(\CC)$ of a lift $\tilde{x}$ of $x$ having a closed orbit.
 Observe that the injective maps $\varphi: \beta^{-1}(x) \to H^1_\tau(\galCR , \mathbf{G}_{\tilde{x}} )$ that we defined in \Cref{subsect:number_of_lifts}, are only surjective for loxodromic elements.
 
 For the topology of $\beta$, we observe that $\beta \vert_\mathcal{L}$ is one-to-one, mapping the punctured open disk $\mathcal{L}$ to $\CC \setminus \bar{\Delta}$.
 As for elliptic orbits in $\mathcal{E}$, the restriction of $\beta$ to each closed triangle of \Cref{fig:espace_he_regions} is  one-to-one to $\bar{\Delta}$. Observe that each segment of the boundary of \Cref{fig:curve_goldman} has exactly two lifts in $X \sslash G$: a segment corresponding to a reflection on a line, contained in $\overline{\mathcal{L}}$ and separating two triangles of $\mathcal{E}$, and a segment in the boundary of $\mathcal{E}$, corresponding to reflexions on points.

 \begin{table}
 	\centering
 	\[
 	\begin{array}{ccc}
 	\toprule
 	\text{Subset containing } x \in \CC & \abs{\beta^{-1}(x)} & \abs{H^1_\tau(\galCR , \mathbf{G}_{\tilde{x}} )} \\
 	\midrule
 	\Delta & 3 & 4 \\
 	\partial \Delta \setminus \{3,3\omega,3\omega^2 \} & 2 & 3 \\
 	\{3,3\omega,3\omega^2 \} & 1 & 2 \\
 	\CC \setminus \bar{\Delta} & 1 & 1 \\
 	\bottomrule
 	\end{array}
 	\]
 	\caption{Number of lifts in $X \sslash G$ of a point $x \in \CC \simeq \mathfrak{X}_{\mathrm{SU}(2,1)}(\ZZ)$, and the cardinal of the corresponding cohomology set $H^1$.}\label{table:nb_orbit_lift_SU21}
 \end{table}

\bibliographystyle{alpha}
\bibliography{character_varieties}

\end{document}